\documentclass[]{siamart0216}

\usepackage{lipsum}
\usepackage{amsfonts}
\usepackage{graphicx}
\usepackage{epstopdf}
\usepackage{algorithmic}
\usepackage{chngpage}
\usepackage{color}
\usepackage{graphicx}
\usepackage{caption}
\usepackage{subcaption}
\ifpdf
  \DeclareGraphicsExtensions{.eps,.pdf,.png,.jpg}
\else
  \DeclareGraphicsExtensions{.eps}
\fi

\newtheorem{prop}{Proposition}

\newtheorem{remark}{Remark}

\newcommand{\TheTitle}{Weak and Strong Solutions to the Inverse-Square Brachistochrone Problem on Circular and Annular Domains} 
\newcommand{\TheAuthors}{Christopher Grimm and John A. Gemmer}

\headers{\TheTitle}{\TheAuthors}

\title{{\TheTitle}}

\author{
  Christopher Grimm\thanks{Division of Applied Mathematics, Brown University, Providence, RI 02906, USA (\email{christopher\textunderscore grimm@brown.edu})}
  \and
  John A. Gemmer\thanks{Division of Applied Mathematics, Brown University, Providence, RI 02906, USA}
}

\usepackage{amsopn}

\ifpdf
\hypersetup{
  pdftitle={\TheTitle},
  pdfauthor={\TheAuthors}
}
\fi

\begin{document}

\maketitle

\begin{abstract}
  In this paper we study the brachistochrone problem in an inverse-square gravitational field on the unit disk. We show that the time optimal solutions consist of either smooth strong solutions to the Euler-Lagrange equation or weak solutions formed by appropriately patched together strong solutions. This combination of weak and strong solutions completely foliates the unit disk. We also consider the problem on annular domains and show that the time optimal paths foliate the annulus. These foliations on the annular domains converge to the foliation on the unit disk in the limit of vanishing inner radius.
\end{abstract}



\section{Introduction}
In 1696 Johann Bernoulli posed the following problem: given two points $A, B$, find a curve connecting $A$ and $B$ such that a particle travelling from $A$ to $B$ under the influence of a uniform gravitational field takes the minimum time.  This is called the \emph{brachistochrone problem}, from the Greek terms \emph{brachistos} for shortest and \emph{chronos} for time.  It was solved the following year by Leibniz, L'Hospital, Newton, and others \cite{dunham1990journey}.  While the solution to the brachistochrone problem has limited applications, the techniques from calculus used to solve it were novel at the time.  Namely, rudimentary techniques from what would later be called the calculus of variations were developed.   Euler and Lagrange later formalized these initial approaches into their celebrated necessary conditions for optimality, what we now call the Euler-Lagrange equations.  For certain types of functionals, this approach reduces the optimization problem to solving a differential equation, i.e. the Euler-Lagrange equation corresponding to the functional.  Indeed, the reduction of an optimization problem to that of solving a differential equation can be directly applied to many classical optimization problems such as the isoperimetric problem \cite{blaasjo2005isoperimetric}, determining the shape of a minimal surface \cite{sagan2012introduction, oprea2000mathematics} and calculating the path of a geodesic on a surface \cite{mccleary2012geometry}.  Moreover, in classical mechanics the dynamics of a system can be derived using the Euler-Lagrange equations to extremize the so called ``action'' of the system \cite{goldstein2014classical}. This approach to classical mechanics is equivalent to Newtonian mechanics but leads to deeper insights which are critical to our current mathematical understanding of quantum mechanics, general relativity, and other branches of physics.

While the Euler-Lagrange equations have been very successfully applied to many problems in engineering and physics they do not provide the complete picture. In particular, as necessary conditions for optimality their derivation implicitly assumes existence and smoothness of a minimum.  In modern mathematics and applications these assumptions are naive.  For instance, many problems in continuum mechanics have minimizers which lack enough regularity to be classical solutions to the Euler-Lagrange equations \cite{muller1999variational}. The existence of these non-standard solutions is not simply a mathematical curiosity but can be realized in practice as the blister and herringbone patterns in compressed thin sheets \cite{ortiz1994morphology, song2008analytical}, branched domain structures in ferromagnets \cite{DeSimone2000},  self-similar patterns in shape memory alloys \cite{bhattacharya2003microstructure}, the network of ridges in crumpled paper \cite{witten2007stress}, and even the fractal like patterns in leaves and torn elastic sheets \cite{audoly2003self, sharon2007geometrically, gemmer2016isometric}. To understand such systems, local solutions of the Euler-Lagrange equations must be ``patched together'' along singularities in a manner that is consistent with the overall variational structure of the problem; see \cite{kohn2006energy} for an introduction to this approach. 

In this paper our focus is more modest. Namely, we study the problem of determining brachistochrone solutions for particles falling in an inverse-square gravitational field. This problem has been studied in \cite{parnovsky1998, Tee99, Gemmer2011} using standard techniques from the calculus of variations. However, in these works they only considered ``strong'' solutions to the Euler-Lagrange equations which limits the scope of the optimal paths considered. In particular, in \cite{parnovsky1998, Tee99, Gemmer2011} it was shown that there is a ``forbidden'' region through which strong solutions to the Euler-Lagrange equations do not penetrate. In this paper, we show that by considering appropriate ``weak'' solutions constructed from strong solutions patched together at the singular origin of the gravitational field, that the full space of optimal paths is more robust. In particular, these solutions enter the forbidden region and are characteristics for the Hamilton-Jacobi equation. This lends credence that our solutions are the natural extensions of the strong solutions that penetrate the forbidden region and are optimal. Moreover, we also consider the inverse-square problem on an annular domain. Using variational inequalities, we show that our weak solutions are obtained in the limit as the inner radius of the annulus vanishes. 

The paper is organized as follows. In Section 2 we outline the general framework of brachistochrone problems, present the strong and weak versions of the Euler-Lagrange equations, and draw a connection to geometric optics using results from optimal control theory. In Section 3 we restrict our focus to the case of inverse-square gravitational field. We first briefly reproduce the results in \cite{parnovsky1998, Tee99}, namely that there exists a forbidden region through which strong solutions cannot penetrate. Next we present our construction of weak solutions that penetrate into this region. In Section 4 we take a pragmatic approach and consider the problem on an annular domain that excises the singularity at the origin. In doing so, we prove that under the assumption that the strong solutions are global minimizers outside of the forbidden region, that our weak solutions are time optimal. We conclude with a discussion section.





\section{Mathematical framework and governing equations}

\subsection{Mathematical framework}
In this section we summarize the essential definitions and equations which we use to study brachistochrone problems in generic settings. First, let $V:\mathbb{R}^{n}\mapsto \mathbb{R}$ be the potential for a gravitational field, i.e. $V$ is a smooth function except possibly at isolated singularities. Suppose $A,B\in \mathbb{R}^n$ satisfy $V(A)>V(B)$ and there exists a smooth curve $\alpha:[0,1] \mapsto \mathbb{R}^n$ satisfying $\alpha(0)=A$, $\alpha(1)=B$ and for all $0 \leq s \leq 1$, $V(\alpha(s))\leq V(A)$. Now, for a particle released in the gravitational field and constrained to fall along $\alpha$, it follows that if friction is neglected mechanical energy is conserved along the path:
\begin{equation}
\left| \alpha^{\prime}(s) \right|^2 \left(\frac{ds}{dt}\right)^2 +V(\alpha(s))=V(A), \label{eqn:Mech_Energy}
\end{equation}
where $t$ denotes time travelled on $\alpha$ and we have absorbed the standard factor of $1/2$ in the kinetic energy into the potential. The total time of flight to $B$ can then be directly computed:
\begin{equation}
T[\alpha]=\int_0^1 \frac{|\alpha^{\prime}(s)|}{\sqrt{V(A)-V(\alpha(s))}}\,ds. \label{eqn:Gen_Tof}
\end{equation}
This time of flight is still well defined if instead of smooth functions we consider \emph{absolutely continuous functions} for which $V(\alpha)\leq V(A)$ \footnote{The space of absolutely continuous functions from $[0,1]$ into $\mathbb{R}^n$ consists of all functions for which there exists a Lebesgue measurable function $\beta:[0,1]\mapsto \mathbb{R}^n$ satisfying
\begin{equation*}
\alpha(s)=\alpha(0)+\int_{0}^s \beta(\bar{s})\,d\bar{s}
\end{equation*}
and is denoted by $AC([0,1];\mathbb{R}^n)$ \cite{leoni2009first}. For $\alpha\in AC([0,1];\mathbb{R}^n)$  the (weak) notion of the first derivative is defined by $\alpha^{\prime}(s)=\beta(s)$.}. That is, we define the \emph{admissible set} $\mathcal{A}$ by
  \begin{equation}
\mathcal{A}=\{ \alpha \in AC\left(\left[0,1\right]; \mathbb{R}^n\right): \alpha(0)=A,\, \alpha(1)=B \text{ and } V(\alpha(s))\leq V(A)\}
\end{equation}
and define the functional $T:\mathcal{A}\mapsto \mathbb{R}$ by Eq. (\ref{eqn:Gen_Tof}). The generalized brachistochrone problem for the potential $V$ is to find a curve $\alpha^{*}\in \mathcal{A}$ that minimizes the time of flight to $B$. We call such curves \emph{brachistochrone solutions} for the the potential $V$.

The contours, i.e. the equipotential curves, of $V$ naturally partition $\mathbb{R}^n$ into domains $U(A)=\{\mathbf{x}\in \mathbb{R}^n: V(A)-V(\mathbf{x})\geq 0\}$ that contain points that (possibly) can be reached by brachistochrone solutions. For example, for the uniform gravitational potential $V:\mathbb{R}^2\mapsto \mathbb{R}$ defined by $V(x,y)=-y$ a particle released at the point $A=(0,0)$ can only reach points in the  set $U(A)=\{(x,y)\in \mathbb{R}^2: y\leq 0\}$. To completely solve the brachistochrone problem for this potential one is naturally led to the question of finding all brachistochrone solutions that foliate $U(A)$. 

\subsection{Euler-Lagrange equations for brachistochrone problems}
We now follow classical techniques presented in \cite{sagan2012introduction} to derive the Euler-Lagrange equations for Eq. (\ref{eqn:Gen_Tof}). First, suppose $\inf_{\alpha\in \mathcal{A}}T[\alpha] <\infty$ and $\alpha^* \in \mathcal{A}$ satisfies $T[\alpha^*]=\inf_{{\alpha}\in \mathcal{A}}T[\alpha]$, i.e. $\alpha^*$ is a minimizer. Since $\alpha^* \in AC([0,1];\mathbb{R}^n)$ and $T[\alpha^*]<\infty$ it follows that the set of points in $[0,1]$ for which $V(\alpha^*(s))=A$ has Lebesgue measure zero. Consequently, if we further assume that $V(\alpha^*(s))=V(A)$ only at $s=0$ and possibly at $s=1$ if the terminal point satisfies $V(B)=V(A)$, then for all $\eta\in C_0^{\infty}([0,1]; \mathbb{R}^n)$\footnote{ $C_0^{\infty}([0,1]; \mathbb{R}^n)$ denotes the space of smooth functions from $[0,1]$ into $\mathbb{R}^n$ with compact support \cite{royden2010real}.} there exists $\bar{h} > 0$ such that $|h|<\bar{h}$ implies $\alpha^{*}+h\eta \in \mathcal{A}$. Define the function $f:[-\bar{h},\bar{h}]\mapsto \mathbb{R}$ by $f(x)=T[\alpha^{*}+h\eta]$. From the regularity assumptions on $V$ and $\alpha^*$ it follows that $f$ is first differentiable in $h$ and consequently since $\alpha^*$ minimizes $T$ it follows that $f^{\prime}(0)=0$. Therefore, we have the following necessary condition for optimality:
\begin{equation}\label{Eq:BrachWeak}
\begin{aligned}
f^{\prime}(0) & =\int_0^{1} \frac{{\alpha^{*}}^{\prime}(s)}{|{\alpha^{*}}^{\prime}(s)| \sqrt{V(A)-V(\alpha^*(s))}}\cdot \eta^{\prime}(s)\,ds 
\\&\,\,\,\,\,\,
+ \frac{1}{2} \int_0^1 \frac{|{\alpha^{*}}^{\prime}(s)|}{(V(A)-V(\alpha^*(s)))^{3/2}} \nabla V (\alpha^*(s))\cdot \eta \,ds,
\end{aligned}
\end{equation}
which must be satisfied for all $\eta\in C_0^{\infty}([0,1];\mathbb{R}^n)$ \cite{evans1998partial}. Eq. (\ref{Eq:BrachWeak}) is known as the \emph{weak formulation of the Euler-Lagrange equations for the brachistochrone problem}. If we further assume that the minimizing curve $\alpha^*$ is second differentiable then Eq. (\ref{Eq:BrachWeak}) can be integrated by parts to yield
\begin{equation}
\begin{aligned}
0 &= \int_0^1\left(\frac{1}{2} \frac{|{\alpha^{*}}^{\prime}(s)|}{(V(A)-V(\alpha^{*}(s)))^{3/2}} \nabla V (\alpha^{*}(s))\right.
\\&\,\,\,\,\,\,
- \left. \frac{d}{ds} \left(\frac{{\alpha^{*}}^{\prime}(s)}{|{\alpha^{*}}^{\prime}(s)| \sqrt{V(A)-V(\alpha^{*}(s))}}\right)\right)\cdot \eta \,ds.
\end{aligned}
\end{equation}
By the so called ``fundamental theorem of the calculus of variations'' it follows that since $\eta$ was arbitrary the necessary condition satisfied by a second differentiable curve $\alpha^{*}$ is the following differential equation \cite{sagan2012introduction}:
\begin{equation}\label{Eq:BrachStrong}
0= \frac{1}{2}\frac{|{\alpha^{*}}^{\prime}(s)|}{(V(A)-V(\alpha^{*}(s)))^{3/2}} \nabla V (\alpha^{*}(s)) -\frac{d}{ds} \left(\frac{{\alpha^{*}}^{\prime}(s)}{|{\alpha^{*}}^{\prime}(s)| \sqrt{V(A)-V(\alpha^{*}(s))}}\right).
\end{equation}
Eq. (\ref{Eq:BrachStrong}) is known as the \emph{strong formulation of the Euler-Lagrange equations for the brachistochrone problem}. 

Note, however, that in deriving the strong formulation of the Euler-Lagrange equations we made the additional assumption that $\alpha^*$ is second differentiable. In many applications this assumption is too restrictive. For example, the functional $J:AC([-1,1];\mathbb{R})\mapsto \mathbb{R}$ defined by $J[y]=\int_{-1}^{1}\left(1-{y^{\prime}(s)}^2\right)^2\,ds$ with boundary conditions $y(-1)=y(1)=1$ is minimized by $y(x)=|x|$. In this example the two strong solutions $y=x$ and $y=-x$ are joined together at $x=0$. However simply gluing together two strong solutions does not guarantee that the resulting combination is a weak solution. If $\alpha^*$ is second differentiable everywhere except at a point $c\in (0,1)$ and satisfies Eq. (\ref{Eq:BrachStrong}) away from $c$, we can integrate Eq. (\ref{Eq:BrachWeak}) by parts to obtain the following necessary condition
\begin{equation}\label{Eqn:WE1}
\lim_{s \rightarrow c^{-}}\left(\frac{{\alpha^{*}}^{\prime}(s)}{|{\alpha^{*}}^{\prime}(s)| \sqrt{V(A)-V(\alpha(s))}}\right)= \lim_{s\rightarrow c^{+}}\left(\frac{{\alpha^{*}}^{\prime}(s)}{|{\alpha^{*}}^{\prime}(s)| \sqrt{V(A)-V(\alpha(s))}}\right).
\end{equation}
Eq. (\ref{Eqn:WE1}) is commonly called the Weierstrass-Erdmann corner condition \cite{sagan2012introduction} and must be satisfied by piecewise smooth solutions of Eq. (\ref{Eq:BrachStrong}).

 We now make some additional comments about the Weierstrass-Erdmann corner conditions which will be relevant to the discussion in later sections. First, away from a singularity in the potential $V$, i.e. if we assume that $(V(A)-V(c))^{-1/2} \neq 0$, Eq. (\ref{Eqn:WE1}) corresponds to continuity of the tangent vector $\alpha^{\prime}(s)$ at $c$. Moreover, away from singularities this condition physically corresponds to conservation of classical momentum at $c$.  However, if $(V(A)-V(c))^{-1/2}=0$ this necessary condition is trivially satisfied. That is, at a singularity in the gravitational field a minimizer could violate conservation of momentum. This result should not be surprising since at a singularity $V(c)=\infty$ implying that the instantaneous speed as well as the acceleration of the particle is infinite. This fact will be critical in our later construction of weak brachistochrone solutions in an inverse-square gravitational field. 

\subsection{Connection with geometrical optics through control theory}
While directly solving the Euler-Lagrange equations given by Eq. (\ref{Eq:BrachStrong}) will solve the brachistochrone problem,  there is another approach that was originally taken by Johann Bernoulli. Namely, Bernoulli realized that the brachistochrone problem is equivalent to finding the path traced out by a ray of light in a medium with index of refraction $n(\mathbf{x})=(V(A)-V(\mathbf{x}))^{-1/2}$. His solution method was prescient in that it applied Snell's law of refraction to what would now be called a finite element approximation to the problem with a piecewise linear basis \cite{erlichson1999johann, sussmann1997300}. The connection to geometrical optics was later exploited by Hamilton and finalized by Jacobi to derive what we now called the Hamilton-Jacobi equations for a variational problem  \cite{broer2014bernoulli, sussmann1997300, nakane2002early}. Specifically, the Hamilton-Jacobi equation is a quasilinear partial differential equation whose characteristic equations are precisely the Euler-Lagrange equations for the system \cite{evans1998partial}. In particular, the Hamilton-Jacobi equation governs the dynamics of wave-fronts propagating in a medium with index of refraction $n(\mathbf{x})$ and the Euler-Lagrange equations are the evolution equations for the normals to the wavefronts. 

We will now show how the geometric optics interpretation of the brachistochrone problem can be directly derived using modern optimal control theory. To reinterpret the brachistochrone problem as a control problem we follow \cite{sussmann1997300} and first define the set of admissible controls by 
\begin{equation}
\mathcal{U}=\{\mathbf{u}:[0,\mathbb{R})\mapsto \mathbb{R}^n: \mathbf{u} \text{ is piecewise smooth and } |\mathbf{u}| =1\}
\end{equation}
and to satisfy Eq. (\ref{eqn:Mech_Energy}) we constrain the dynamics of the system by
\begin{equation} \label{Eq:Trajectory}
\dot{\alpha}=\sqrt{V(A)-V(\mathbf{x})}\mathbf{u}.
\end{equation}
We define the \emph{trajectory of a control} to be the curve $\alpha$ defined by Eq. (\ref{Eq:Trajectory}) and also define $\mathcal{T}_B:\mathcal{U}\mapsto \mathbb{R}\cup \{+\infty\}$ to be the first time a trajectory corresponding to a control $\mathbf{u}$ reaches the point $B$.  The optimal control problem corresponding to the brachistochrone problem is to find $\mathbf{u}^*\in \mathcal{U}$ that steers a trajectory $\alpha(t)$ to a point $B\in U(A)$ in the minimal amount of time. That is, find $\mathbf{u}^*\in \mathcal{U}$ such that $T_B[\mathbf{u}^*]=\inf_{\mathbf{u} \in \mathcal{U}}\mathcal{T}_B[\mathbf{u}]$. Clearly, this optimal control problem is equivalent to our previous formulation of the brachistochrone problem and $\inf_{\mathbf{u} \in \mathcal{U}}\mathcal{T}_B[\mathbf{u}]=\inf_{\alpha \in \mathcal{A}} T[\alpha]$.

One technique for solving such an optimal control problem is to apply Bellman's technique of dynamic programming \cite{bertsekas1995dynamic}. Namely, if we define the \emph{value function} $\mathcal{V}:U(A)\mapsto \mathbb{R}$ by 
\begin{equation}
\mathcal{V}(\mathbf{x})=\inf_{\mathbf{u}\in \mathcal{U}}\mathcal{T}_{\mathbf{x}}[\mathbf{u}] 
\end{equation}
then the dynamic programming principle states that for $\Delta t>0$ sufficiently small
\begin{equation}
\mathcal{V}(\mathbf{x})=\min_{\substack{\mathbf{u}\in \mathcal{U}\\ 0<s<\Delta t}}\left \{ \mathcal{V}(\bar{\alpha}(\Delta t))+\Delta t \right\},
\end{equation}
where $\bar{\alpha}$ corresponds to the \emph{time reversed} solution of Eq. (\ref{Eq:Trajectory}) with initial condition $\bar{\alpha}(0)= \mathbf{x}$ and control $\mathbf{u}$. If we assume that $\mathcal{V}$ is smooth we can formally Taylor expand:
\begin{equation*}
\mathcal{V}(\mathbf{x})=\min_{|\mathbf{u(0)}|=1}\left \{\mathcal{V}(\mathbf{x}) +\nabla \mathcal{V}(\mathbf{x}) \sqrt{V(A)-V(\mathbf{x})}\mathbf{u}(0)\Delta t +\Delta t +\mathcal{O}(\Delta t^2) \right \}.
\end{equation*}
Consequently, taking the limit as $\Delta t \rightarrow 0$ we obtain
\begin{equation*}
-1 =\min_{|\mathbf{u(0)}|=1}\nabla \mathcal{V}(\mathbf{x}) \sqrt{V(A)-V(\mathbf{x})}\mathbf{u}(0).
\end{equation*}
Finally, this minimum is obtained by $\mathbf{u}(0)=-\nabla \mathcal{V}(\mathbf{x})/|\mathcal{V}(\mathbf{x})|$ and hence we can conclude that  the value function $\mathcal{V}$ satisfies the following partial differential equation
\begin{equation}\label{Eq:Eikonal}
|\nabla \mathcal{V}|^2=\frac{1}{V(A)-V(\mathbf{x})}=n^2(\mathbf{x}).
\end{equation}

In geometrical optics, Eq. (\ref{Eq:Eikonal}) is an eikonal equation for a medium with index of refraction $n(\mathbf{x})$. That is, the level sets of solutions to Eq. (\ref{Eq:Eikonal}) correspond to wave fronts for light traveling through the medium and the light rays correspond to curves that are everywhere tangent to the normals of the level sets of $\mathcal{V}$. Consequently, if we let $\beta(s)$ be an arc-length parameterization of such a light ray it follows that
\begin{equation}\label{Eq:Ray1}
\nabla \mathcal{V}(\beta(s))=n(\beta(s)) \frac{d \beta}{ds}.
\end{equation}
Differentiating with respect to $ds$ and switching the order of differentiation, it follows that
\begin{equation*}
\nabla \left( \nabla \mathcal{V}(\beta(s)) \cdot \frac{d \beta}{ds}\right)=\frac{d}{ds} \left( n(\beta(s)) \frac{d\beta}{ds}\right).
\end{equation*}
Therefore, it follows from Eq. (\ref{Eq:Eikonal}) and (\ref{Eq:Ray1}) that the governing equation for the rays is 
\begin{equation}\label{Eq:RayEquation}
\nabla n(\beta)= \frac{d}{ds}\left( n(\beta(s)) \frac{d\beta}{ds}\right).
\end{equation}
Finally, it follows immediately that Eq. (\ref{Eq:RayEquation}) is simply a version of Eq. (\ref{Eq:BrachStrong}) that is parametrized by arc-length. That is, to solve the brachistochrone problem we could, in principle, solve the eikonal equation (\ref{Eq:Eikonal}) and use Eq. (\ref{Eq:Ray1}) to reconstruct the brachistochrone solution. More importantly, we could solve the brachistochrone problem directly by solving the Euler-Lagrange equations (\ref{Eq:BrachStrong}) and compute the time of flight along these solution curves to compute the solution to the eikonal equation (\ref{Eq:Eikonal}). 

\section{Brachistochrone Problem in an Inverse-Square Gravitational Field}

\subsection{Framework}
Consider two points $A, B \in \mathbb{R}^2$. For an inverse-square field, the potential $V: \mathbb{R}^2 \mapsto \mathbb{R}$ is defined by $V(x) = -|x|^{-2}$. The inverse-square brachistochrone problem is to construct a curve connecting $A$ and $B$ such that a particle traversing the curve from $A$ to $B$ under the influence a gravitational field centered at the origin with potential $V$ has the least time of flight. In this case the admissable set $\mathcal{A}_0$ is defined by
\begin{equation} \label{eq:admissible_set}
\mathcal{A}_{0}=\left\{\alpha\in AC([0,1]; \mathbb{R}^2): \alpha(0)=A, \alpha(1)=B \text{ and } \forall t \in [0,1], |\alpha(t)|\leq |A|\right\}
\end{equation}
and the time of flight $T : \mathcal{A}_0\mapsto \mathbb{R}^+$ is given by: 
\begin{equation} \label{eq:tof}
    T[\alpha] = \int_{0}^{1}{\frac{|\alpha^\prime(s)|}{\sqrt{|\alpha(s)|^{-2} - |A|^{-2}}} ds}. 
\end{equation}
Again, this functional arises from classical conservation of mechanical energy and the constraint $|\alpha(s)| \leq |A|$ -- a necessary condition for this functional to be well defined -- is equivalent to the condition that the particle cannot gain mechanical energy. 

To study minimizers of Eq. (\ref{eq:tof}) it is natural to work in a polar-coordinate representation of the form
\begin{equation} \label{eq:alpha_polar}
\alpha(s) = \left(r\left(s\right)\cos\left(\theta(s)\right), r\left(s\right)\sin\left(\theta(s)\right)\right),
\end{equation}
where $r: [0, 1] \to \left[0, |A|\right]$ and $\theta: [0, 1] \to [-\pi, \pi]$ are (weakly) differentiable functions satisfying $r(0) = |A|$, $r(1) = |B|$, $\theta(0) = \theta_0$, $\theta(1) = \theta_f$, with $\theta_0$, $\theta_f$ the angular coordinates of $A, B$ respectively; see Fig. \ref{fig:basic_curve}(a). By rotational symmetry and radial invariance we can assume without loss of generality that $A = (1, 0)$; see Fig. \ref{fig:basic_curve}(b). In this representation Eq. \eqref{eq:tof} becomes
\begin{equation} \label{eq:T_functional}
T[r, \theta] = \int_{0}^{1}{\sqrt{\frac{r^\prime(s)^2 + r(s)^2\theta^\prime(s)^2}{r(s)^{-1}- 1}} ds} = \int_{0}^{1}{L_2\left(r(s), r^\prime(s), \theta^\prime(s)\right) ds},
\end{equation}
where $L_2:\mathbb{R}^3 \to \mathbb{R}$ denotes the Lagrangian for this functional. To reduce encumbering notation we write $(r(s),\theta(s))\in \mathcal{A}_0$ as a proxy for the statement that there exists $\alpha \in \mathcal{A}_0$ with corresponding radial and angular components $r(s)$ and $\theta(s)$ respectively. 

\begin{figure}[ht!]
\centering
\begin{subfigure}[b]{0.45\textwidth}
\centering
\includegraphics[height=.75\textwidth]{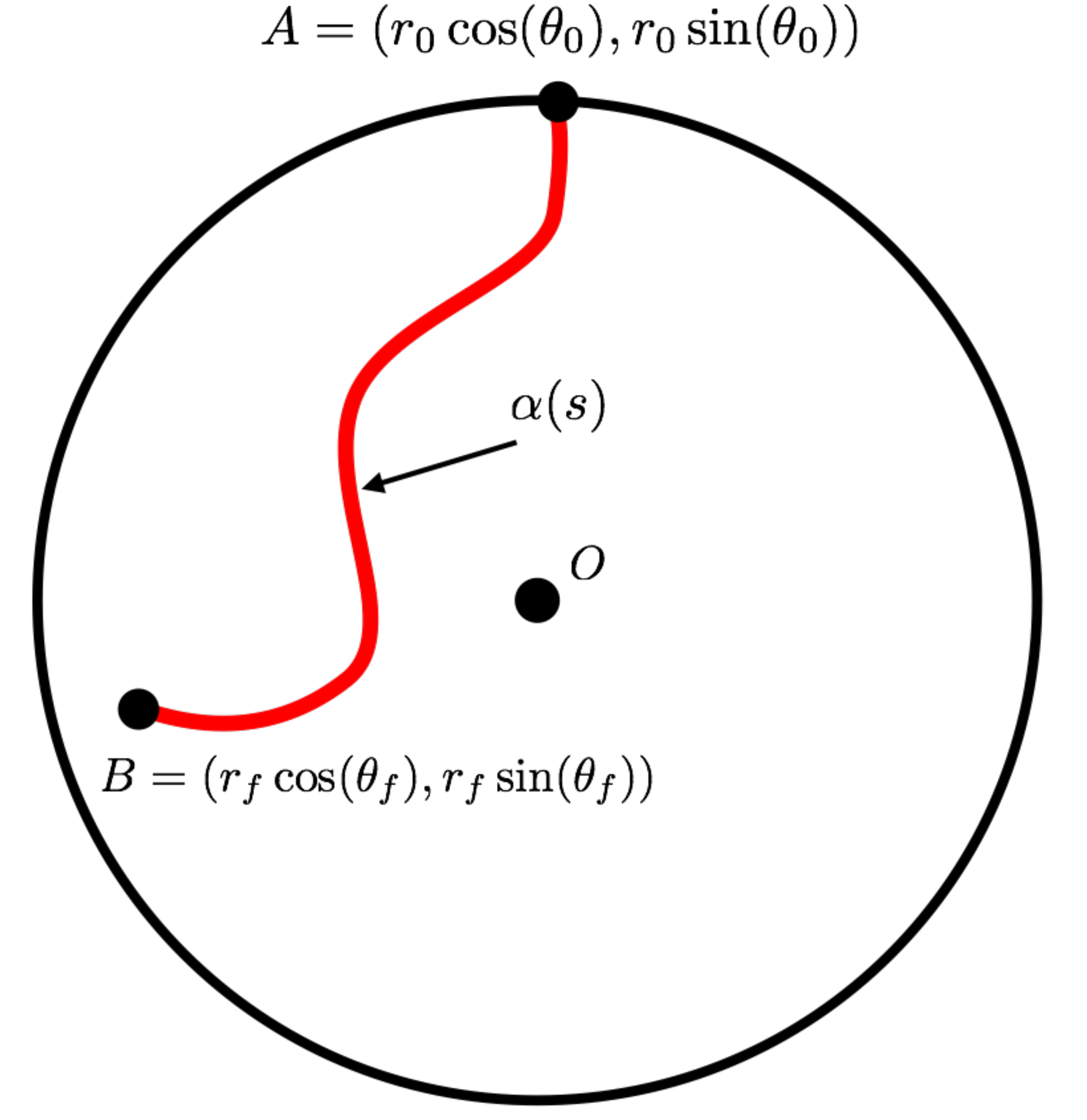}
\caption{}
\end{subfigure}
~
\begin{subfigure}[b]{0.45\textwidth}
\centering
\includegraphics[height=.75\textwidth]{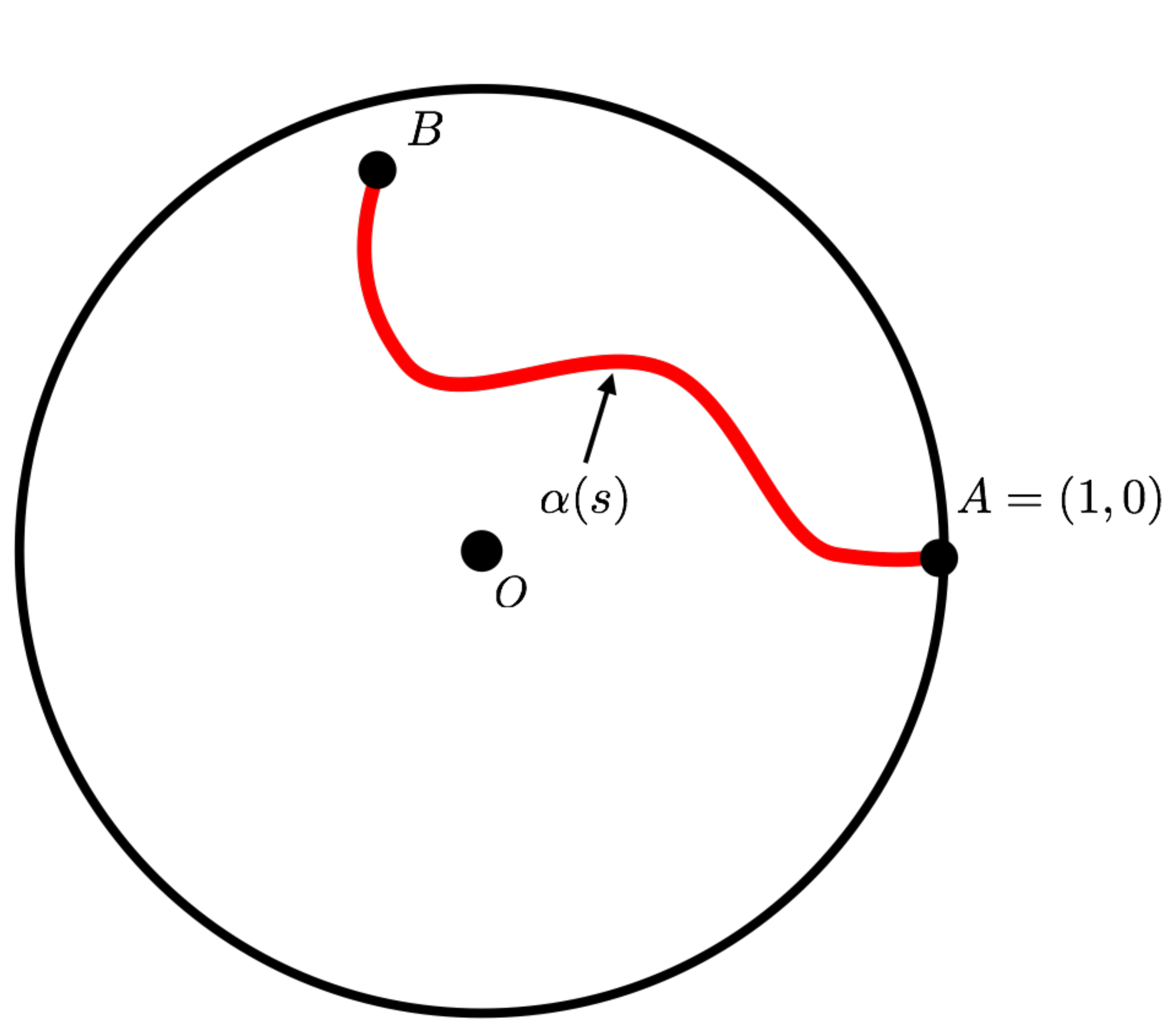}
\caption{}
\end{subfigure} 
\caption{(a) Plot of a curve $\alpha:[0,1]\mapsto \mathbb{R}^2$ connecting $A=(r_0\cos(\theta_0),r_0\sin(\theta_0))$ to $B=(r_f\cos(\theta_f),r_f\sin(\theta_f))$ in an inverse-square gravitational field centered at the origin $O$. The circle of radius $r_0$ is an equipotential for the inverse-square gravitational field. For a particle falling along this curve, conservation of mechanical energy requires that $|\alpha(s)|\leq r_0$. (b) By rotational and radial scale invariance of this problem we can assume without loss of generality that $A=(1,0)$.}   \label{fig:basic_curve}
\end{figure}

We now deduce geometric properties of minimizers using the structure of the Lagrangian. We first show that if $(r^*,\theta^*)\in \mathcal{A}_0$ minimizes $T$ then $\theta^*$ must be a monotone function. This property prevents minimizers from ``turning back'' to its starting point. The idea of the proof is to construct for all $(r,\theta)\in \mathcal{A}_0$ a modified curve $(r,\overline{\theta})\in \mathcal{A}$  with $\overline{\theta}$ monotone in $s$ and show $T[r,\overline{\theta}]\leq T[r,\theta]$.  

\begin{prop}  \label{prop:monotonicity} If $(r^*(s),\theta^*(s))\in \mathcal{A}_0$ minimizes $T$ then $\theta^*$ is monotone in $s$.
\end{prop}
\begin{proof} Let $(r(s),\theta(s))\in \mathcal{A}_0$ terminate at the point $(r_f\cos(\theta_f),r_f\sin(\theta_f))$ and assume $\theta_f\geq 0$. Define $\overline{\theta}:[0,1]\mapsto [0,2\pi]$ by $\overline{\theta}(s)=\min\{\theta_f,\sup\{\theta(t): 0\leq t \leq s\}\}$. From absolutely continuity of $\theta$ it follows that $\overline{\theta}$ is absolutely continuous, monotone increasing and satisfies $\overline{\theta}(1)=\theta_f$.  Moreover, there exists a closed set $I$ on which $\overline{\theta}=\theta$ and on open set $\overline{I}=[0,1]\setminus I$ on which $\frac{d\overline{\theta}}{ds}=0$. Therefore,
\begin{align*}
T[r,\overline{\theta}]&=\int_I   \sqrt{ \frac{r^{\prime}(s)^2+r(s)^2\theta^{\prime}(s)^2}{r(s)^{-1}-1}}\,ds+\int_{\overline{I}} \sqrt{ \frac{r^{\prime}(s)^2}{r(s)^{-1}-1}}\,ds            
\\
& \leq \int_I   \sqrt{ \frac{r^{\prime}(s)^2+r(s)^2\theta^{\prime}(s)^2}{r(s)^{-1}-1}}\,ds+\int_{\overline{I}}   \sqrt{ \frac{r^{\prime}(s)^2+r(s)^2\theta^{\prime}(s)^2}{r(s)^{-1}-1}}\,ds\\
&= T[r,\theta],
\end{align*}
with equality if and only if $\theta$ is monotone increasing. Consequently, if $(r^*(s),\theta^*(s))\in \mathcal{A}_0$ minimizes $T$ then $\theta^*$ is monotone increasing in $s$. A similar argument proves that $\theta^*$ must be monotone decreasing if $\theta_f<0$. 
\end{proof}

We now prove that without loss of generality we can assume minimizers must be symmetric about the angle $\theta_f/2$. Specifically, if in polar coordinates $(r^*(s),\theta^*(s))\in \mathcal{A}_0$ minimizes $T$ and terminates at the final point $r_f=1$, $\theta_f$ then the image of $(r(s),\theta(s))$ is symmetric about the line $\theta=\theta_f/2$. Similar to the previous proof, the idea is to modify a curve $\alpha\in \mathcal{A}_0$ by constructing symmetric versions and comparing the time of flight. 

\begin{prop} \label{prop:symmetry} If $(r^*(s),\theta^*(s))\in \mathcal{A}_0$ minimizes $T$ with terminal point $r(1)=1$, $\theta(1)=\theta_f$ then there exists a version of $(r^*(s),\theta^*(s))$ in $\mathbb{R}^2$ that is symmetric about the line $\theta=\theta_f/2$ that also minimizes $T$.
\end{prop}
\begin{proof}
Let $(r(s),\theta(s))\in \mathcal{A}_0$ and $t(s)$ a reparameterization in which $\theta(1/2)=\theta_f/2$ and if $t>1/2$ then $\theta(t)>\theta_f/2$. Define the two possible reflections of $(r(t),\theta(t))$ about the line $\theta=\theta_f/2$ by
\begin{align*}
r_1(t)&=\begin{cases}
r(t), & 0\leq t \leq \frac{1}{2}\\
r(1-t), & \frac{1}{2} < t \leq 1
\end{cases} \text{ and } \theta_1(t)= \begin{cases}
\theta(t), & 0\leq t \leq \frac{1}{2}\\
\theta_f-\theta(1-t), & \frac{1}{2} < t \leq 1
\end{cases}\\
 r_2(t)&=\begin{cases}
r(1-t), & 0\leq t \leq \frac{1}{2}\\
r(t), & \frac{1}{2} < t \leq 1
\end{cases} \text{ and } \theta_2(t)= \begin{cases}
\theta(1-t)-\theta_f & 0\leq t \leq \frac{1}{2}\\
\theta(t), & \frac{1}{2} < t \leq 1
\end{cases}.
\end{align*}
By construction the images of the curves $(r_1(t),\theta_1(t)$, $(r_2(t),\theta_2(t))$ in $\mathbb{R}^2$ are symmetric about the line $\theta=\theta_f/2$. It follows from symmetry that
\begin{align*}
T[r_1,\theta_1]&=2\int_{0}^{1/2}{\sqrt{\frac{r^{\prime}(s)^2 + r(s)^2\theta^\prime(s)^2}{r(s)^{-1}- 1}} ds} \\
T[r_2,\theta_2]&=2\int_{1/2}^{1}{\sqrt{\frac{r^{\prime}(s)^2 + r(s)^2\theta^\prime(s)^2}{r(s)^{-1}- 1}} ds}. 
\end{align*}
Therefore, $T[r_1,\theta_1]+T[r_2,\theta_2]=2 T[r,\theta]$ from which it follows that
\begin{equation*}
\min\{ T[r_1,\theta_1], T[r_2,\theta_2]\} \leq T[r,\theta]. 
\end{equation*}
Therefore it follows that if $(r^*(s),\theta^*(s))\in \mathcal{A}_0$ minimizes $T$ then either $(r_1(s), \theta_1(s))$ or $(r_2(s), \theta_2(s))$ must also minimize $T$.  
\end{proof}

\subsection{Strong Solutions to Euler-Lagrange Equations}

In this subsection we review the construction of smooth minimizers to $T$ that was originally presented in \cite{parnovsky1998, Tee99}. The classic method for finding time minimizing curves is to derive the Euler-Lagrange equations for $T$ and solve the resulting boundary value problem. Specifically, if we assume there exists a second differentiable curve $\alpha^*(s) \in \mathcal{A}_0$ with angular component $\theta^*(s)$ and radial component $r^*(s)$ that (locally) minimizes $T$ then the resulting boundary value problem is:

\begin{equation} \label{boundary}
    \begin{cases}
    \left. \left(\frac{\partial L}{\partial r} - \frac{d}{ds}\frac{\partial L}{\partial r^\prime} \right) \right|_{(r^*(s), \theta^*(s))} = 0 \\
    \left. \frac{d}{ds} \frac{\partial L}{\partial \theta^\prime} \right|_{(r^*(s) \theta^*(s))} = 0 \\
        r^*(0) = 1, r^*(1) = |B| \\
        \theta^*(0) = 0, \theta(1) = \theta_f
    \end{cases}.
\end{equation}

If we make the assumption that $\theta^{*}$ is a function of $r^*$, i.e. we assume the ansatz $r^*(s) = (|B| - 1)s + 1$, then Eq. \eqref{boundary} reduces to the differential equation: 
\begin{equation}\label{eq:theta_diff_eq}
    \frac{d}{dt}\frac{\partial L}{\partial\theta^{\prime}} = 0.
\end{equation}
Formally, Eq. (\ref{eq:theta_diff_eq}) can be integrated to yield a separable differential equation with solution
\begin{equation} \label{theta_eq}
    \theta^*(r*) = \pm\int_{1}^{r^*}{\sqrt{\frac{2(\frac{1}{u} - 1)D}{u^4 - u^2(2(\frac{1}{u} - 1))D}}du},
\end{equation}
where $D > 0$ is a constant of integration which can be determined from the boundary conditions. The assumption that $\theta^*$ is globally a function of $r^*$ is valid if $\left(\frac{d\theta^*}{dr^*}\right)^{-1}\neq 0$ for all $r^*\in(0,1)$. It follows from Eq. (\ref{theta_eq}) that for fixed $D > 0$ this condition is equivalent to the non-existence of solutions to the equation $r^3+2Dr-2D=0$ for $r\in(0,1]$. The following proposition makes this statement precise and identifies the critical radius $r_c$ in terms of the integration constant $D$. 

\begin{prop} \label{prop:bijection}
For fixed $D>0$, there exists unique $r_c(D)\in [0,1]$ such that $\theta^{*}(r^*)$ defined by Eq. (\ref{theta_eq}) satisfies
\begin{equation*}
\lim_{r\rightarrow r_c(D)^+} \left.\frac{d\theta^*}{dr^*}\right|_{r}=\infty.
\end{equation*}
Moreover, the mapping $D\mapsto r_c(D)$ is a bijection from $(0,\infty)$ into $(0,1)$. 
\end{prop}
\begin{proof} 
Fix $D>0$ and define $g:(0,1)\mapsto \mathbb{R}$ by $g(r)=r^3+2Dr-2D$. Since $g(0)=-2D$, $g(1)=1$, and $g^{\prime}(r)>0$ it follows from the Intermediate Value Theorem there exists unique $r_c(D)\in (0,1)$ such that $g(r_c(D))=0$. Consequently, it follows from Eq. (\ref{theta_eq}) that
\begin{equation*}
\lim_{r \to r_c(D)^+}\left.{\frac{\partial \theta^*}{\partial r^*}}\right|_{r} = \infty.
 \end{equation*}
   The bijection is proved by noting that the inverse mapping from $r_c(D)$ to $D$ given by $D(r_c)=r_c^3/2(1-r_c)$ satisfies $\lim_{r_c\rightarrow 0^+}D(r_c)=0$, $\lim_{r_c\rightarrow 1^{-}}D(r_c)=\infty$ and is monotone increasing in $r_c$.
 \end{proof}

To extend smooth solutions beyond the point where $\theta^*$ is no longer a function of $r^*$ it follows from Proposition \ref{prop:symmetry} that it is necessary to reflect the solutions across the line $\theta=\theta_f/2$.  Specifically, for $D\in (0,\infty)$ if we define $r_c(D)$ as in Proposition \ref{prop:bijection} then we obtain the following family of solutions expressed in parametric form $\alpha(s)=(r_D^S(s)\cos(\theta_D^S(s)),r_D^S(s)\sin(\theta_D^S(s)))$ with:
\begin{align}
r_D^S(s)&=\begin{cases}
2(r_c(D)-1)s+1, & 0\leq s \leq \frac{1}{2}\\
2(1-r_c(D))(s-1)+1, & \frac{1}{2} < s \leq 1 
\end{cases}, \label{eq:rparamaterization}\\
\theta_D^S(s)&=\begin{cases}
\displaystyle{\pm \int_{1}^{r_D^S(s)}\sqrt{\frac{2(1 - u)D}{u^5 - 2u^2D(1-u)}  }\,du}, & 0\leq s\leq \frac{1}{2}\\
\displaystyle{\mp \int_{r_c(D)}^{r_D^S(s)}\sqrt{\frac{2(1 - u)D}{u^5 - 2u^2D(1-u)}  }\,du}\pm \theta_D^S(1/2), & \frac{1}{2}\leq s\leq 1
\end{cases}, \label{eq:thparamaterization}
\end{align}
where we are using the superscript ``$S$'' to denote that these are strong solutions to the Euler-Lagrange equations. Note, that while the individual functions $r_D^S(s)$ and $\theta_D^S(s)$ are not smooth, the curve $\alpha$ itself is a smooth function from $[0,1]$ into $\mathbb{R}^2$. 

\begin{figure}[htb]
\centering 
\includegraphics[width=.5\textwidth]{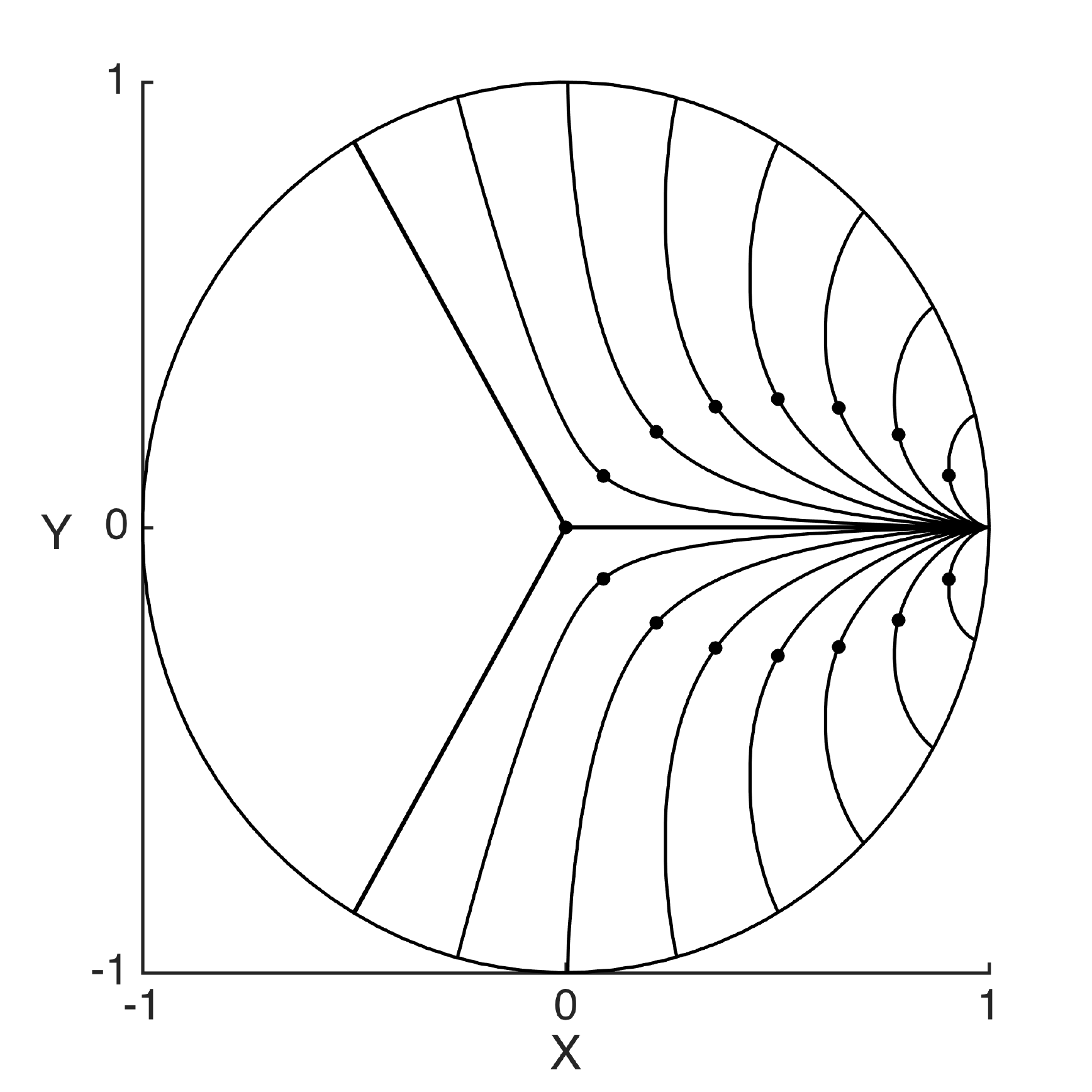}
\caption{Plot of sixteen smooth strong solution curves $\alpha(s)$ defined by Eqs. (\ref{eq:rparamaterization}-\ref{eq:thparamaterization}) with the final angular coordinate  $\theta_D^S(1)$ uniformly spaced from $-2\pi/3$ to $2\pi/3$. The value of $D$ was found by the bisection method, i.e. the shooting method, applied to Eq. (\ref{eq:thparamaterization}). The points indicate the critical radius $r_c(D)$ where $\frac{d\theta_D^S}{dr_D^S} = \pm \infty$ and the curve begins receding away from the origin.}\label{fig:classic_solutions}
\end{figure}

Interestingly, as $D$ ranges over values in $(0, \infty)$ the curves defined by Eqs. (\ref{eq:rparamaterization}) and (\ref{eq:thparamaterization}) do not foliate the unit disk ${x^2 + y^2 \leq 1}$; see Fig. \ref{fig:classic_solutions}. Indeed, if we define the sector $S$ by 
\begin{equation}
S=\{\theta : -\pi \leq \theta < -2\pi /3\} \cup \{\theta: 2\pi /3 < \theta < \pi\}
\end{equation}
it was shown in \cite{Tee99, Gemmer2011} that these curves do not enter $S$. This is made precise by the following proposition whose proof we adapt from \cite{Gemmer2011}.  

\begin{prop} \label{prop:characterization1}
For all $s \in [0, 1]$ and $D \in (0, \infty)$ the curves $\alpha(s)$ with radial and angular components $r_D^S(s)$ and $\theta_D^S(s)$ defined by Eqs. (\ref{eq:rparamaterization}) and (\ref{eq:thparamaterization}) satisfy $\theta_D^S(s) \notin S$.
\end{prop}

\begin{proof}
    For $D > 0$ let $\alpha(s) = (r_D^S(s)\cos(\theta_D^S(s)), r_D^S(s)\sin(\theta_D^S(s)))$ be defined by Eqs. (\ref{eq:rparamaterization}) and (\ref{eq:thparamaterization}) with the ``-'' branch. Differentiating, it follows that
\begin{equation*}
    \frac{d\theta_D^S}{ds} = \frac{d\theta_D^S}{dr_D^S}\frac{dr_D^S}{ds} = 2r_c(D)\sqrt{\frac{2(1-r_D^S)D}{(r_D^S)^5 - 2(r_D^S)^2D(1-r_D^S)}} > 0
\end{equation*}
with equality only at $r_D^S = 1$. Hence, $\frac{d\theta_D^S}{ds}$ is monotone increasing in $s$ with the maximum angular coordinate $\bar{\theta}(D)$ satisfying
\begin{equation*}
    \bar{\theta}(D) = \max_{0 \leq s \leq 1}{\theta^S_D(s)} = 2\int_{r_c(D)}^{1}{\sqrt{\frac{2(1-u)D}{u^5 - 2u^2D(1-u)}}du}.
\end{equation*}
Since $\lim_{D \to \infty}{r_c(D)} = 1$ it follows that $\lim_{D \to \infty}{\bar{\theta}(D)} = 0$. Now, from uniquess of solutions to Eq. (\ref{eq:theta_diff_eq}) we can deduce that $\bar{\theta}(D)$ must be monotone decreasing in $D$ and hence has a limit as $D \to 0$. By making the change of variables $x = \frac{r_c(D)}{u}^{\frac{3}{2}}$, we obtain
\begin{align*}
    \frac{1}{2}\lim_{D \to 0}{\bar{\theta}(D)} &= \lim_{D \to 0}{\frac{2}{3}\int_{r_c(D)}^{1}{\sqrt{\frac{1 - r_c(D)x^{-\frac{2}{3}}}{(1 - r_c(D)) - (1 - r_c(D)x^{-\frac{2}{3}})x^2}}dx}} \\
                                               &= \lim_{D \to 0}{\frac{2}{3}\int_{0}^{1}{\sqrt{\frac{1}{\frac{1 - r_c(D)}{1 - r_c(D)x^{-\frac{2}{3}}} - x^2}}\mathbb{I}\{x > r_c(D)^\frac{3}{2}}\}dx}
\end{align*}
where $\mathbb{I}$ denotes the standard indicator function. Now, observing that the integrand of the above equation forms a sequence of functions bounded by $(1 - x^2)^{-\frac{1}{2}}$ it follows from Lebesgue's Dominated Convergence Theorem that
\begin{equation*}
    \lim_{D \to 0}{\bar{\theta}(D)} = \int_{0}^{1}{\frac{4}{3}\sqrt{\frac{1}{1 - x^2}}dx} = \frac{4}{3}(\arcsin(1) - \arcsin(0)) = \frac{2\pi}{3}.
\end{equation*}
The exact same arguments hold if we consider the ``+'' branch in Eqs. (\ref{eq:rparamaterization}) and (\ref{eq:thparamaterization}) except the limiting angle is $-\frac{2\pi}{3}$. Consequently we can conclude that for all $s \in [0, 1]$ and $D \in (0, \infty)$ that $\theta_D^S(s)$ given by Eq. (\ref{eq:thparamaterization}) satisfies $-\frac{2\pi}{3} \leq \theta_D^S(s) \leq \frac{2\pi}{3}$.
\end{proof}

The following proposition immediately follows from Propositions \ref{prop:bijection} and \ref{prop:characterization1}.
\begin{prop} \label{prop:characterization2}
    For all $\theta_f \in (0, 2\pi/3)$ there exists $D \in (0, \infty)$ such that the solution curve with angular and radial coordinates $(\theta_D^S(s), r_D^S(s))$ as defined by Eqs. (\ref{eq:rparamaterization}-\ref{eq:thparamaterization}), satisfies $\theta_D^S(0) = 0$, and $\theta_D^S(1) = \theta_f$. Moreover, the mapping $\theta_f \mapsto D$ is a bijection.
\end{prop}

\begin{remark}
    It follows from Propositions \ref{prop:bijection} and \ref{prop:characterization1} that $\theta_f$, $r_c$ and $D$ characterize a unique solution curve of the form defined by Eqs. (\ref{eq:rparamaterization}-\ref{eq:thparamaterization}).
\end{remark}
\subsection{Weak Solutions to the Euler-Lagrange Equations}
One family of weak solutions to Eq. (\ref{eq:tof}) is given by the following parameterization:
\begin{align}
  r^W(s)&=\begin{cases}
        1 - 2s, & 0\leq s \leq \frac{1}{2}\\
        2s - 1, & \frac{1}{2} < s \leq 1 
    \end{cases}, \label{eq:rweak}\\
   \theta^W_{\theta_f}(s)&=\begin{cases}
        0, & 0\leq s\leq \frac{1}{2}\\
        \theta_f, & \frac{1}{2}\leq s\leq 1
    \end{cases}, \label{eq:tweak}
\end{align}
where the superscript ``$W$'' is used to denote that these are weak solutions to the Euler-Lagrange equations. This parameterization indeed satisfies the corner condition given by Eq. (\ref{Eqn:WE1}):

\begin{align*}
&\lim_{s \rightarrow c^{-}}\left(\frac{{\alpha^{*}}^{\prime}(s)}{|{\alpha^{*}}^{\prime}(s)| \sqrt{V(A)-V(\alpha(s))}}\right) - \lim_{s\rightarrow c^{+}}\left(\frac{{\alpha^{*}}^{\prime}(s)}{|{\alpha^{*}}^{\prime}(s)| \sqrt{V(A)-V(\alpha(s))}}\right)\\
&= (1, 0) \lim_{s \to \frac{1}{2}^+}\left(\sqrt{\frac{r^W(s)}{r^W(s) - 1}}\right) + (\cos (\theta_f), \sin(\theta_f)) \lim_{s \to \frac{1}{2}^-}\left(\sqrt{\frac{r^W(s)}{r^W(s) - 1}}\right) \\
&= (0, 0).
\end{align*}
It is important to note that as it is defined $\theta_{\theta_f}^W(s)$ is not weakly first differentiable. Specifically, $\theta_{\theta_f}^W(s)$ is only differentiable in the distributional sense with a derivative given by a delta mass centered at $s=1/2$. However, this is only an artifact of the $r=0$ coordinate singularity for polar coordinates and the curve $\alpha^W(s)$ itself is weakly differentiable. Moreover, for $s<1/2$ this curve is simply the solution curve given by Eqs. (\ref{eq:rparamaterization}-\ref{eq:thparamaterization}) with $D=0$ and the weak solution is constructed by joining appropriately rotated copies of this strong solution at the origin. 

The family of solutions to the Euler-Lagrange equations defined by Eqs. (\ref{eq:rweak}-\ref{eq:tweak})  completely foliate the unit disk; see Fig. \ref{fig:complete_foliation}(a). Hence these solution curves are natural candidates for time minimizers that enter the sector $S$. 
In Fig. \ref{fig:complete_foliation}(b) we plot the unit disk foliated by a combination of strong and weak solutions to the Euler-Lagrange equations. More specifically, for a given $\theta_f$ we use Eqs. (\ref{eq:rparamaterization}-\ref{eq:thparamaterization}) or Eqs. (\ref{eq:rweak}-\ref{eq:tweak}) depending on whether $| \theta_f | > 2\pi /3$. The contour beneath the curves in corresponds to the time of flight computed along the solution curves and confirms our intuition that the classic solutions have shorter time of flight outside of $S$. Notice that the contours in Fig. \ref{fig:complete_foliation}(b) are smooth and intersect the strong and weak solution curves orthogonally as expected from Eq. (\ref{Eq:Ray1}). Moreover, the value function $\mathcal{V}$ defined in Section 2 is a solution to the eikonal equation defined by Eq. (\ref{Eq:Eikonal}).

\begin{figure}[ht!]
\centering
\begin{subfigure}[b]{0.45\textwidth}
\centering
\includegraphics[height=.9\textwidth]{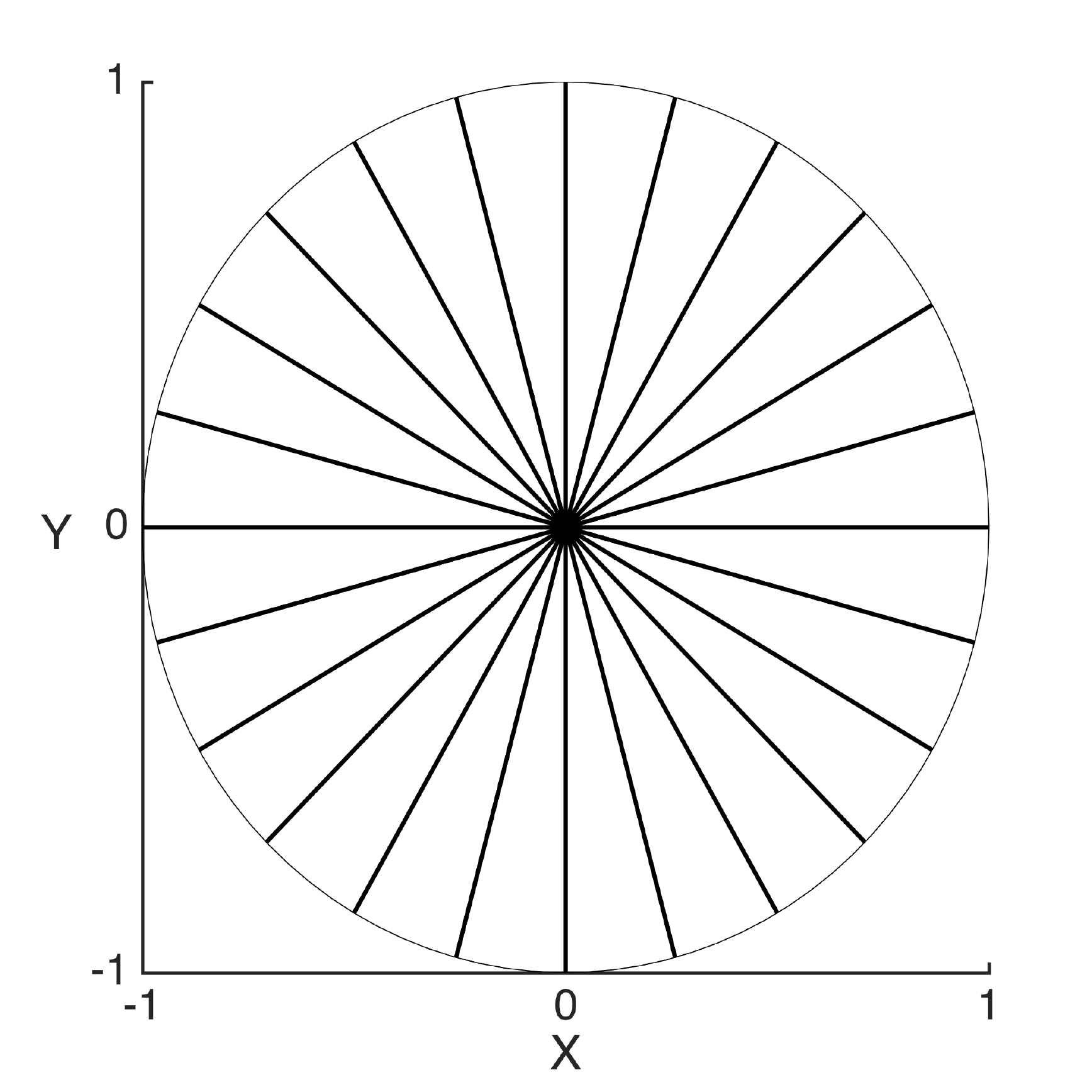}
\caption{}
\end{subfigure}
~
\begin{subfigure}[b]{0.45\textwidth}
\centering
\includegraphics[height=.9\textwidth]{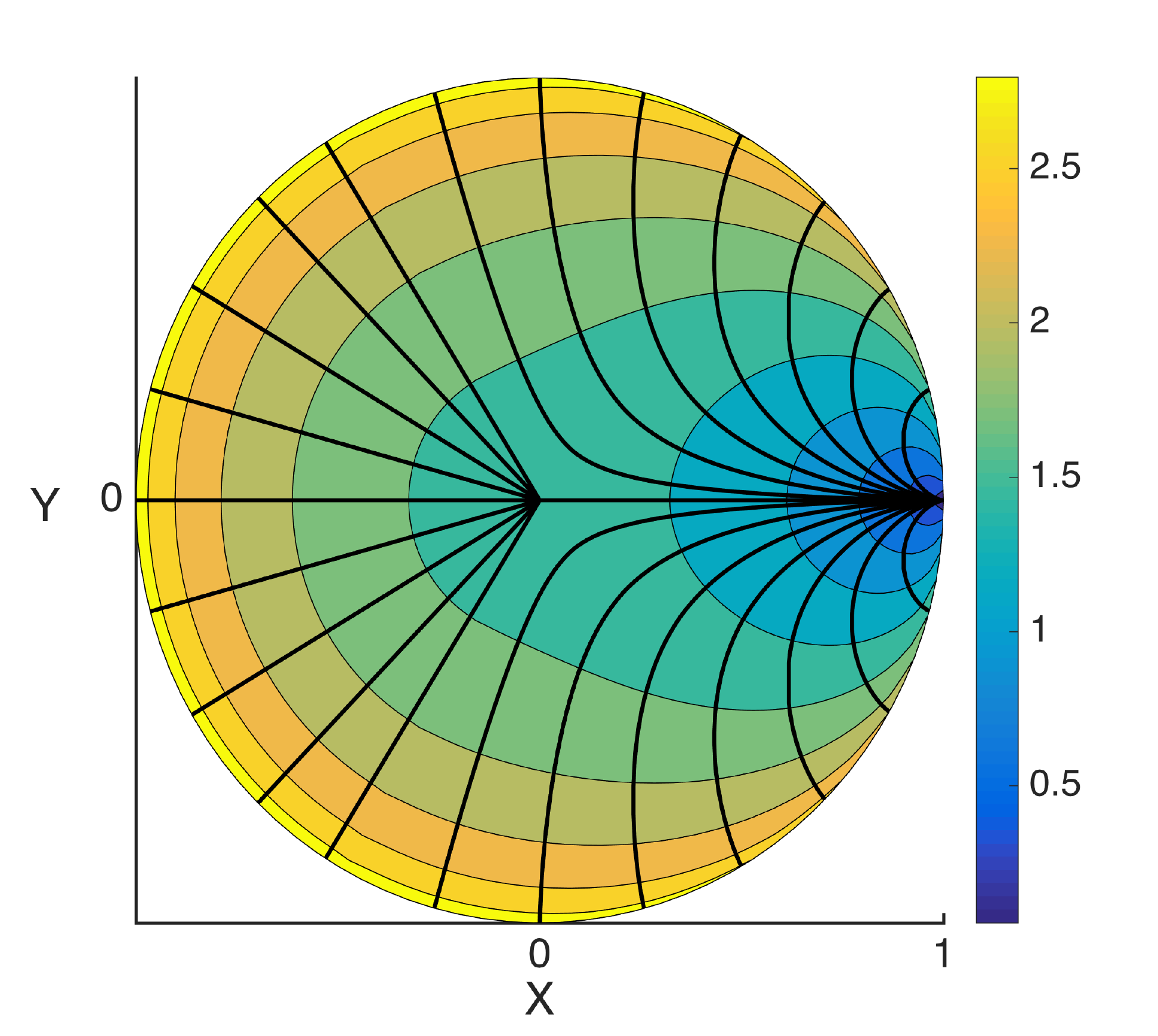}
\caption{}
\end{subfigure} 
\caption{(a) Plot of weak solution curves defined by Eqs. (\ref{eq:rweak}-\ref{eq:tweak}) with final angular coordinate $\theta_f$ spaced evenly from $-\pi$ to $\pi$. (b) Foliation of the unit disk by weak solution curves in $S$ and classic solutions outside of $S$. Beneath the solution curves is a contour plot of the time of flight calculated along the solution curves. The contours rorrespond to level sets of the value $\mathcal{V}$ satisfying the eikonal equation defined by Eq. (\ref{Eq:Eikonal}).} \label{fig:complete_foliation}
\end{figure}

\section{Constrained Inverse-Square Brachistochrone Problem}
\subsection{Variational Inequality}
In the previous section we solved the inverse-square brachistochrone problem using a combination of weak and strong extremizers. However, the solutions are impractical in that as a consequence of the singular gravitational field a particle following along an extremizer will experience infinite acceleration at the origin. To alleviate this problem we now consider a modified  version of the inverse-square brachistochrone problem that restricts the radial coordinate to remain bounded away from the origin. Specifically, for $\epsilon>0$ we define the annulus $\mathcal{O}_\epsilon= \{(x, y) \in \mathbb{R}^2 : \sqrt{x^2 + y^2} \geq \epsilon\}$ and consider the problem of minimizing $T$ over the  admissible set $\mathcal{A}_{\epsilon}\subset \mathcal{A}_0$ defined by
\begin{equation}
\mathcal{A}_{\epsilon} = \left\{(r(s),\theta(s)) \in \mathcal{A}_0 : \left(r(s)\cos(\theta(s),r(s)\sin(\theta(s))\right) \in \mathcal{O}_\epsilon \text{ for } s \in [0, 1]\right\}.
\end{equation}
This formulation of the problem is equivalent to an ``obstacle problem'' with the  obstacle being the circle of radius $\epsilon$ centered at the origin. 

To derive necessary conditions satisfied by minimizers of $T$ over $\mathcal{A}_{\epsilon}$ we follow Ref. \cite{evans1998partial} Chapter 8, Section 4 and derive a variational inequality that plays the role of the Euler-Lagrange equations. First, suppose $\alpha^*(s)\in \mathcal{A}_{\epsilon}$ is the global minimizer of $T$ over $\mathcal{A}_{\epsilon}$ with radial and angular components $r^*(s)$ and $\theta^*(s)$ respectively. Letting  $\beta(s)\in \mathcal{A}_{\epsilon}$ with radial and angular components $q(s)$ and $\theta^*(s)$ respectively it follows from convexity of $\mathcal{A}_{\epsilon}$ that for all $\lambda\in [0,1]$ the curve $\gamma(s)$ with radial component $r^*(s)+\lambda(q(s)-r^*(s))$ and angular component $\theta^*(s)$ also satisfies $\gamma(s)\in \mathcal{A}_{\epsilon}$. Consequently
\begin{equation}
T\left[r^*(s)+\lambda(q(s)-r^*(s)), \theta^*(s)\right]-T\left[r^*(s),\theta^*(s)\right]\geq 0
\end{equation}
and thus taking the limit $\lambda \to 0$ we obtain the following necessary condition satisfied by a minimizer:
\begin{equation} \label{eq:ObstacleEL1}
    \int_{0}^{1}\left((q(s) - r^*(s))\left.\frac{\partial{L}}{\partial{r}}\right|_{r^*(s), \theta^*(s)} + (q^{\prime}(s) - {r^*}^{\prime}(s))\left.\frac{\partial{L}}{\partial{r^\prime}}\right|_{r^*(s), \theta^*(s)}\right)ds \geq 0.
\end{equation}
Since we can perturb $\theta^*(s)$ by any smooth function $\xi$ compactly supported on $[0,1]$ we again obtain the following weak Euler-Lagrange equation:
\begin{equation}\label{eq:ObstacleEL2}
    \int_{0}^{1} \xi^\prime(s)\left.\frac{\partial{L}}{\partial{\theta^\prime}}\right|_{r^*(s), \theta^*(s)}ds =0.
\end{equation}

We now illustrate how Eqs. (\ref{eq:ObstacleEL1}-\ref{eq:ObstacleEL2}) can be used to derive further necessary conditions satisfied by a minimizer $(r^*(s),\theta^*(s))\in \mathcal{A}_{\epsilon}$. Suppose $(r^*(s), \theta^*(s)) \in \mathcal{A}_\epsilon$ minimizes $T$ over $\mathcal{A}_\epsilon$ and define the following sets
\begin{align}
    U &= (r^*)^{-1}\{\epsilon\}, \label{eq:U-Set}\\
    U^c &= [0,1]\setminus U. \label{eq:V-Set}
\end{align}
Since $r^*(s)$ is continuous it follows that $U$ is a closed subset of $[0,1]$. On $U$ it follows that Eq. (\ref{eq:ObstacleEL1}) is automatically satisfied since
\begin{equation*}
\begin{aligned}
   &\int_U\left((q(s) - r^*(s))\left.\frac{\partial{L}}{\partial{r}}\right|_{r^*(s), \theta^*(s)} + (q^{\prime}(s) - {r^*}^{\prime}(s))\left.\frac{\partial{L}}{\partial{r^\prime}}\right|_{r^*(s), \theta^*(s)}\right)ds \\
   &\hspace{70mm}= \int_U (q(s)-\epsilon) \frac{|{\theta^*} ^{\prime}(s)|}{1-\epsilon} \frac{3-2\epsilon}{2}\,ds \geq 0.
\end{aligned}
\end{equation*}
On $U^c$ consider the perturbation $q(s)=\tau v(s)+r^*(s)$, where $v$ is any smooth function compactly supported on $V$ and $\tau\in \mathbb{R}$ is small enough in magnitude that $q(s)\in \mathcal{A}_\epsilon$. Substituting into Eq. (\ref{eq:ObstacleEL1}) yields 
\begin{equation} \label{eq:RecoveringEL}
    \tau\int_{U^c}{\left(v(s)\left.\frac{\partial{L}}{\partial{r}}\right|_{r^*(s), \theta^*(s)} + v'(s)\left.\frac{\partial{L}}{\partial{r'}}\right|_{r^*(s), \theta^*(s)}ds\right)} \geq 0.
\end{equation}
Since $\tau$ is of arbitrary sign the above inequality is actually an equality. That is, for $s\in U^c$ we recover the weak Euler-Lagrange equations for $r^*(s)$. 

\begin{remark}
Taken together the above necessary conditions imply that potential minimizers of $T$ over $\mathcal{A}_{\epsilon}$ consist of the family of curves satisfying the Euler-Lagrange equations away from the constraint. That is, potential minimizers consist of piecewise smooth curves satisfying Eqs. (\ref{eq:rparamaterization}-\ref{eq:thparamaterization}) joined with circular arcs of radius $\epsilon$.
\end{remark}

\subsection{Piecewise Smooth Minimum}
As in the case with no constraint, i.e. $\epsilon=0$, we now foliate $\mathcal{O}_{\epsilon}$ by curves that minimize the time of flight. By symmetry we only foliate the upper half-annulus  $\mathcal{O}_\epsilon^+ = \{(x, y) \in \mathcal{O}_\epsilon : y \geq 0\}$. To construct the foliation we examine the behavior of potential minimizers with terminal coordinates $(r_t,\theta_f)$ satisfying $(r_f,\theta_f) \in \partial \mathcal{O}_{\epsilon}^+$ (the boundary of $\mathcal{O}_{\epsilon}^+$) which can be naturally divided into four regions:
\begin{align*}
R_1 &= \{(r, \theta) \in \partial\mathcal{O}_\epsilon^+ : \theta=0\},\\
R_2 &= \{(r, \theta) \in \partial\mathcal{O}_\epsilon^+ : r = 1\}, \\
R_3 &= \{(r, \theta) \in \partial\mathcal{O}_\epsilon^+ : \theta = \pi \},\\
R_4 &= \{(r, \theta) \in \partial\mathcal{O}_\epsilon ^+: r = \epsilon \}.
\end{align*} 
Each of these regions is considered as separate cases below.

\subsubsection{Minimizers terminating on $R_1$}
It follows from Eqs. (\ref{eq:rparamaterization}-\ref{eq:thparamaterization}) that if $D=0$ the strong solution to the Euler-Lagrange equation is a straight line connecting $(1,0)$ to the origin. In particular this implies that if $(r_f,\theta_f)\in R_1$ then straight lines are the natural candidate minimizers.

\subsection{Minimizers terminating on $R_2$}
Suppose $(r_f,\theta_f)\in R_2$. For $\theta_f$ sufficiently small we expect the minimizers to consist of the smooth strong solution curves defined by Eqs. (\ref{eq:rparamaterization}-\ref{eq:thparamaterization}). However, if $\theta_f>\pi/3$  the strong solutions to the Euler-Lagrange equations will necessarily intersect the obstacle. Note that from convexity of the strong solutions there exists a unique critical angle $\theta_c^{\epsilon}\in (0,\pi/3)$ in which the strong solutions intersect the obstacle tangentially. Specifically, $\theta_c^{\epsilon}$ is defined by
\begin{equation}
\left(r^S_{D(\epsilon)}(1/2),\theta^S_{D(\epsilon)}(1/2)\right)=(\epsilon,\theta_c^\epsilon) \text{  and } \left.\frac{dr_{D(\epsilon)}^S}{ds}\right|_{1/2}=0.
\end{equation}
The critical angle $\theta_c^{\epsilon}$ serves as a boundary in the sense that if the final angular coordinate $\theta_f$ satisfies $\theta_f/2 >\theta_c^{\epsilon}$ then it is necessary to consider piecewise defined curves as candidate minimizers. This is made precise by the following proposition.

\begin{prop} \label{prop:minimizeAwayFromConstraint}
Suppose that the smooth solution curves given by Eqs. (\ref{eq:rparamaterization}-\ref{eq:thparamaterization}) are global minimizers of $T$ over $\mathcal{A}_0$. For $\epsilon>0$  if $\theta_f/2<\theta_c^{\epsilon}$ then there exists a $D \geq D(\epsilon) > 0$ such that $(r_D^S(s),\theta_D^S(s))\in \mathcal{A}_{\epsilon}$ minimizes  $T$ over curves in  $\mathcal{A}_{\epsilon}$ terminating at the angular coordinate $\theta_f$. 
\end{prop}
\begin{proof}
Let $\theta_f \in (0, 2\pi/3)$ satisfy $\theta_f/2 < \theta_c^\epsilon$ and $(r_D^S(s), \theta_D^S(s))$ parameterize the smooth solutions given by Eq. (\ref{eq:rparamaterization}-\ref{eq:thparamaterization}) terminating at $(\cos(\theta_f), \sin(\theta_f))$. From Propositions \ref{prop:characterization1} and \ref{prop:characterization2} it follows that $\theta_f$ and $r_c$ are monotonically decreasing and increasing in $D$ respectively, and thus $r_c(D(\theta_f)) \geq \epsilon$. Furthermore it follows that since $r_D^S(s)$ is convex in $s$ that $r_D^S(s)\geq \epsilon$ and thus $(r_D^S(s), \theta_D^S(s)) \in \mathcal{A}_\epsilon$. Finally, since $\mathcal{A}_\epsilon \subset \mathcal{A}_0$ and $(r_D^S(s), \theta_D^S(s))$ is assumed to minimize $T$ over curves in $\mathcal{A}_0$ which terminate at $(\cos(\theta_f), \sin(\theta_f))$, it follows that $(r_D^S(s), \theta_D^S(s))$ also minimizes $T$ over curves in $\mathcal{A}_\epsilon$ which terminate at $(\cos(\theta_f), \sin(\theta_f))$. 
\end{proof}

For a strong solution given by Eqs. (\ref{eq:rparamaterization}-\ref{eq:thparamaterization}) satisfying $\theta_f/2  >\theta_c^\epsilon$, let $s_D^\epsilon= \min\{(r_D^S)^{-1}\{\epsilon\}\}$, i.e. the first point of intersection with the obstacle. The natural generalizations of Propositions \ref{prop:monotonicity} and \ref{prop:symmetry} can be shown to hold for $\mathcal{A}_\epsilon$ and consequently we know, without loss of generality, that minimizers consist of curves symmetric about the angle $\theta_f/2$ that are smooth solutions given by Eqs. (\ref{eq:rparamaterization}-\ref{eq:thparamaterization}) away from the constraint, ride along it for a finite amount of time, and then rejoin a rotated and reflected version of the latter half of the same smooth solution, see Fig. \ref{fig:family_samples}. This family of minimizers is given below:

\begin{equation} \label{eq:obstacleFamily}
    \mathcal{F}_{D,\theta_f}^{\epsilon}(s)= \begin{cases}
        (r_{D}^S(h(s))\cos(h(s))), r_{D}^S(h(s)))\sin(\theta_{D}^S(h(s))))) & s \in [0, 1/3)\\
        (\epsilon\cos(t(s)), \epsilon\sin(t(s))) & s \in [1/3, 2/3]\\
        (r_{D}^S(s)\cos(l(s)), r_{D}^S(s)\sin(l(s))) & s \in (2/3, 1]
    \end{cases},
\end{equation}
where $r_D^S(s)$, $\theta_D^S(s)$ are given by Eqs. (\ref{eq:rparamaterization}-\ref{eq:thparamaterization}) with $\theta_f(D)$ satisfying $\theta_f(D)/2 \geq \theta_c^{\epsilon}$, $h(s) = s/(3s_D^\epsilon)$, $j(s) = 3s_D^\epsilon s + 1 - 3s_D^\epsilon$, $l(s) = \theta_D^S(j_{D, \epsilon}(s)) + \theta_f - \theta_D^S(1)$ and $t(s) = 3(\theta_f - 2\theta^S(s_D^\epsilon))s + \theta_f - \theta^S(s_D^\epsilon)$. The following proposition characterizes the minimum of $T$ over the family of curves given by Eq. (\ref{eq:obstacleFamily}), namely they consist of  the curves defined by Eq. (\ref{eq:obstacleFamily}) that meet the constraint at a tangent. 

\begin{figure}[htb]\label{fig:FamilySamples}
    \centering
    \includegraphics[width=1.0\textwidth]{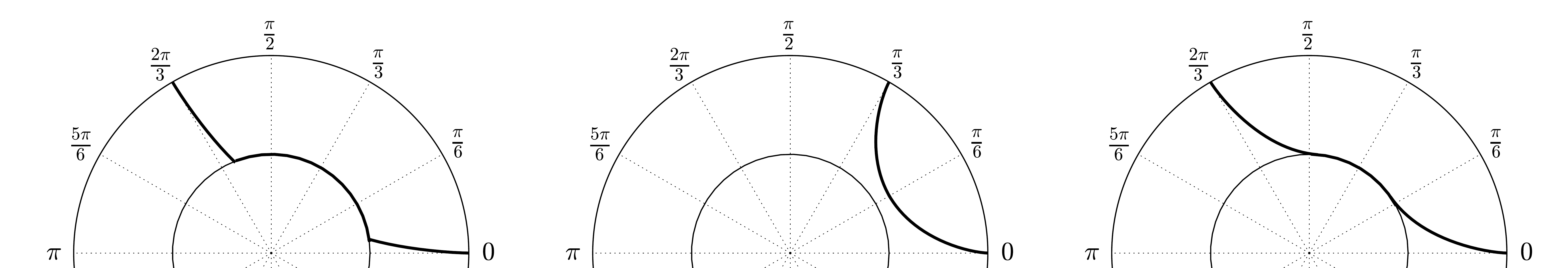}
    \caption{Plots of curves $\mathcal{F}_{D, \theta_f}^\epsilon(s)$ defined by Eq. (\ref{eq:obstacleFamily}) with $\epsilon = 0.5$. (a) A curve $\mathcal{F}_{D, \theta_f}^\epsilon(s)$ with $\theta_f = 2\pi/3$, $D = 0.0204$ and $\theta_c = \pi/4$ which reaches the obstacle and rides along it. (b) A curve $\mathcal{F}_{D, \theta_f}^\epsilon(s)$ with $\theta_f = \pi/3$, $D = 0.2300$ and $\theta_c = \pi/6$ which does not reach the obstacle ($s_\epsilon = 0.5$). (c) A curve $\mathcal{F}_{D, \theta_f}^\epsilon(s)$ with $\theta_f = 2\pi/3$, $D = 0.1250$ and $\theta_c = 0.5981$ which approaches the obstacle at a tangent and rides along it.} \label{fig:family_samples}
\end{figure}

\begin{prop} \label{prop:minimizeAlongConstraint}
	Suppose that the smooth solution curves given by Eqs. (\ref{eq:rparamaterization}-\ref{eq:thparamaterization}) are global minimizers in $\mathcal{A}_0$. For $\epsilon>0$  if $\theta_f/2 \geq \theta_c^{\epsilon}$ then the unique minimizer of $T$ over the family of curves defined by Eq. (\ref{eq:obstacleFamily}) intersects the constraint tangentially.  
\end{prop}
\begin{proof}
Let $\theta_f \in (0, \pi)$ satisfy $\theta_f/2 \geq \theta_c^\epsilon$. Let $\mathcal{F}_{D(\epsilon), \theta_f}^\epsilon(s)$ be the unique curve which intersects the constraint at a tangent and terminates at angular coordinate $\theta_f$. Let $\mathcal{F}_{D, \theta_f}^\epsilon(s)$ be another curve with $D < D(\epsilon)$ that terminates at angular coordinate $\theta_f$. From Propositions \ref{prop:characterization1} and \ref{prop:characterization2} it follows that $\theta_f$ and $r_c$ are monotonically decreasing and increasing in $D$ respectively, hence $r_c(D) \leq \epsilon$ and $\theta_c(D) \geq \theta_c^\epsilon$. Moreover, it follows from the monotonicity of $\theta_D^S(s)$ in $s$ that $\mathcal{F}_{D, \theta_f}^\epsilon(s)$ intersects the constraint at some angular coordinate $\theta_0 < \theta_c^\epsilon$ and intersects $\mathcal{F}_{D(\epsilon), \theta_f}^\epsilon(s)$ at angular coordinate $\theta_c^\epsilon$. Since $\mathcal{F}_{D(\epsilon), \theta_f}^\epsilon(s)$ is of the form $(r_{D(\epsilon)}^S(s), \theta_{D(\epsilon)}^S(s))$ for  $s < s_{D(\epsilon)}^\epsilon$, it follows from our assumption that smooth solutions given by Eqs. (\ref{eq:rparamaterization}-\ref{eq:thparamaterization}) minimize $T$ that $\mathcal{F}_{D(\epsilon), \theta_f}^\epsilon(s)$ minimizes the time-of-flight to angular coordinate $\theta_c^\epsilon$. Moreover, both curves have the same time of flight along the constraint from angular coordinates $\theta_c^\epsilon$ to $\theta_f - \theta_c^\epsilon$, and consequently $T[\mathcal{F}_{D(\epsilon), \theta_f}^\epsilon(s)] < T[\mathcal{F}_{D, \theta_f}^\epsilon(s)]$. 

\end{proof}



\subsubsection{Minimizers Terminating on $R_3$}
Let $(r, \pi) \in R_3$. We know from our prior analysis in $R_2$ that all minimizers must ride along the obstacle until at least angular coordinate $\pi - \theta_c^\epsilon$. Hence, to minimize over curves terminating at $(r, \pi)$, we need only minimize $T$ over curves from $(\epsilon, \pi - \theta_c^\epsilon)$ to $(r, \pi)$. Note that from the convexity of strong solutions given by Eqs. (\ref{eq:rparamaterization}-\ref{eq:thparamaterization}), there exists a unique angle $\phi_c^\epsilon \geq \pi - \theta_c^\epsilon$ such that the smooth solution comes off the obstacle tangentially at $\phi_c^\epsilon$ and intersects $(r, \pi)$. It can be further shown that this curve minimizes the time of flight $T$ between coordinates $(\epsilon, \pi - \theta_c^\epsilon)$ and $(r, \pi)$. This is made precise by the following proposition.

\begin{prop} \label{prop:R4minimizer}
Suppose that the smooth solution curves given by Eqs. (\ref{eq:rparamaterization}-\ref{eq:thparamaterization}) are global minimmizers in $\mathcal{A}_0$. For $\epsilon > 0$ and $(r, \pi) \in R_3$, a minimizer of $T$ over $\mathcal{A}_{\epsilon}$ that terminates at $(r,\pi)$ leaves the obstacle at a tangent. 
\end{prop}
\begin{proof}
Let $\alpha^*\in \mathcal{A}_{\epsilon}$ denote the candidate minimizer that leaves the obstacle at a tangent from the angular coordinate $\phi_c^{\epsilon}$ and terminates at the polar coordinate $(r,\pi)$. From the convexity of smooth solutions we know that any other candidate minimizer $\alpha\in \mathcal{A}_{\epsilon}$  terminating at the polar coordinate $(r,\pi)$ must come off the obstacle at some angular coordinate $\phi_0 > \phi_c^\epsilon$. Since, the piece of the curve $\alpha^*$ that connects the polar coordinates $(\epsilon,\phi_c^{\epsilon})$ to $(r,\pi)$ is a smooth solution curve given by Eqs. (\ref{eq:rparamaterization}-\ref{eq:thparamaterization}) it follows from our assumption that it minimizes the time of flight from polar coordinates $(\epsilon, \phi_c^{\epsilon})$ to $(r,\pi)$. Consequently, the time of flight from polar coordinates $(\epsilon,\phi_c^{\epsilon})$ to $(r,\pi)$ along $\alpha$ is larger. 

 Using the same method as in the proof of Proposition \ref{prop:minimizeAlongConstraint}, it can be shown that it follows from our assumption that smooth solutions given by Eqs. (\ref{eq:rparamaterization}-\ref{eq:thparamaterization}), that a candidate minimizer coming off the obstacle at angular coordinate $\phi_0$ has greater time of flight than one coming off at angular coordinate $\phi_c^\epsilon$.
\end{proof}

\subsubsection{Minimizers terminating on $R_4$}

Let $(\epsilon, \theta) \in R_4$. If $\theta < \theta_c^\epsilon$, there is a smooth solution given by Eqs. (\ref{eq:rparamaterization}-\ref{eq:thparamaterization}) that, by assumption, minimizes the time of flight to $(\epsilon, \theta)$. If $\theta \in (\theta_c^\epsilon, \pi - \theta_c^\epsilon)$, there exists a curve in the family Eq. (\ref{eq:obstacleFamily}) that minimizes the time of flight to $(\epsilon, \theta)$. If $\theta > \theta_c^\epsilon$, it follows from Proposition \ref{prop:R4minimizer} that there is a curve minimizing the time of flight that approaches the obstacle at a tangent and rides along until angular coordinate $\theta$. 

\begin{remark}
The solution curves connected to the boundary foliate the domain. Specifically, there are three distinct areas $A_1$, $A_2$ and $A_3$ satisfying $\mathcal{O}_\epsilon = A_1 \cup A_2 \cup A_3$ that are foliated by curves connected to the boundary in $R_2$, $R_3$ and $R_4$ respectively. This is illustrated in Figure \ref{fig:wedges}.
\end{remark}

\begin{figure}[ht!]
\centering
\begin{subfigure}[b]{0.45\textwidth}
\centering
\includegraphics[width=\textwidth]{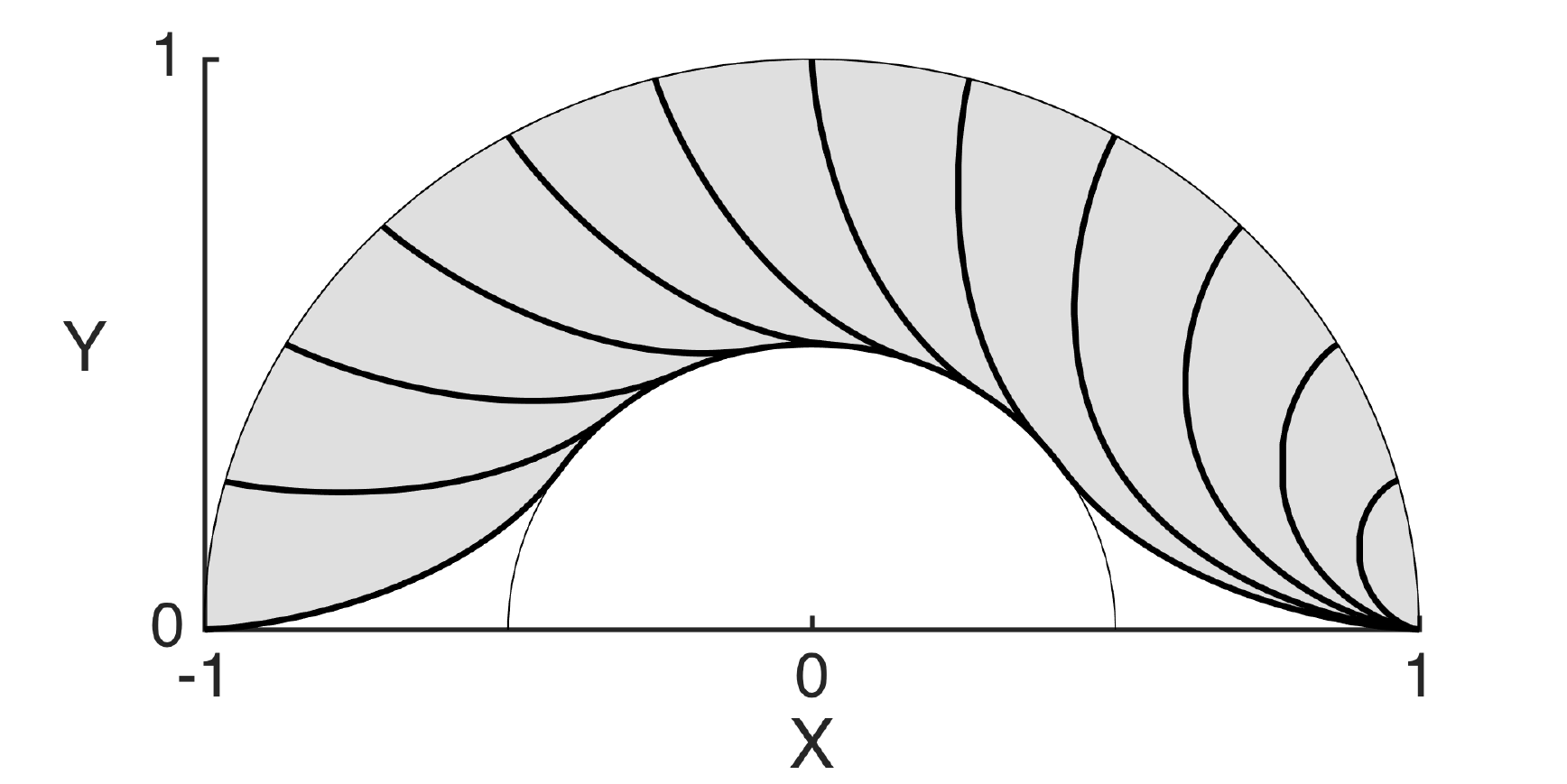}
\caption{}
\end{subfigure}
~
\begin{subfigure}[b]{0.45\textwidth}
\centering
\includegraphics[width=\textwidth]{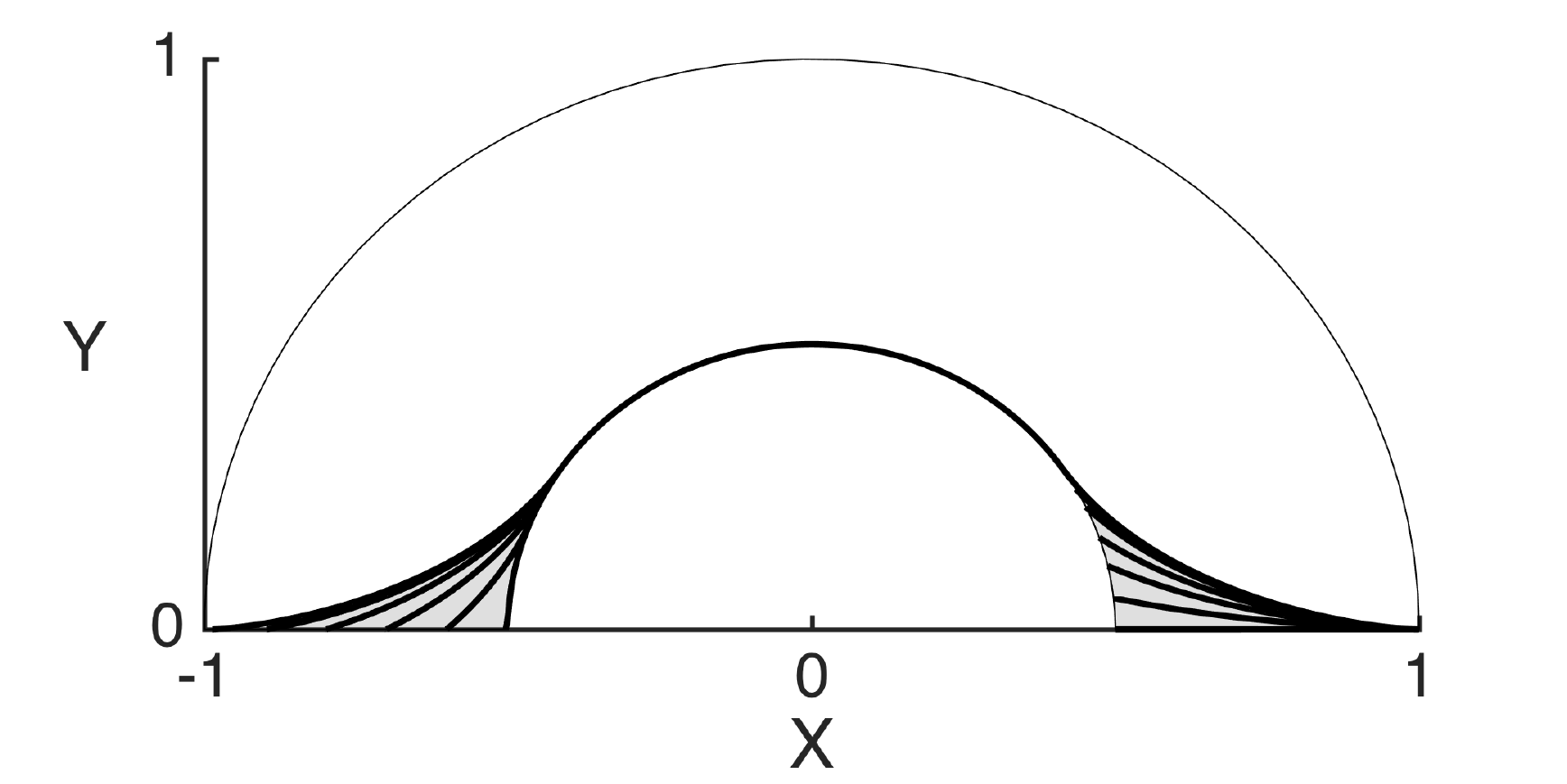}
\caption{}
\end{subfigure} 
\caption{(a) The annulus $\mathcal{O}_\epsilon$ with $A_1$ shaded in. $A_1$ consists of the set of points which lie on solution curves terminating on $R_2$. Overlaid on $A_1$ are evenly spaced solution curves terminating on $R_2$ given by Eq. (\ref{eq:obstacleFamily}). (b) The annulus $\mathcal{O}_\epsilon$ with $A_2$, $A_3$ shaded in. $A_2$ and $A_3$ consist of the sets of points which lie on solution curves terminating on $R_3$ and $R_4$ respectively. Overlaid on $A_2$ and $A_3$ are evenly spaced solution curves terminating in $R_3$ and $R_4$.}\label{fig:wedges}
\end{figure}



\subsection{Convergence to weak solutions}
In the previous section, Propositions \ref{prop:minimizeAwayFromConstraint} and \ref{prop:minimizeAlongConstraint} describe the behavior of a family of curves that minimizes $T$ to terminal polar coordinates $(1, \theta_f) \in R_2$. For a given value of $\epsilon \in (0, 1)$ and $\theta_f \in (0, \pi)$, we denote this family as 
\begin{equation} \label{eq:optimalConstrainedSolution}
\alpha_{\theta_f}^\epsilon(s) = \begin{cases}
(r^S_{D(\theta_f)}(s)\cos(\theta^S_{D(\theta_f)}(s)), r^S_{D(\theta_f)}(s)\sin(\theta^S_{D(\theta_f)}(s))) & \text{ if } \theta_f/2 \leq \theta_c^\epsilon  \\
\mathcal{F}_{D(\epsilon), \theta_f}^\epsilon(s) & \text{ if } \theta_f/2 >  \theta_c^\epsilon \end{cases}
\end{equation}
where $(r^S_{D(\theta_f)}(s), \theta^S_{D(\theta_f)}(s))$ are the radial and angular coordinates of the unique smooth solution given by Eqs. (\ref{eq:rparamaterization}-\ref{eq:thparamaterization}) and $\mathcal{F}_{D(\epsilon), \theta_f}^\epsilon(s)$ is a member of the family described by Eq. (\ref{eq:obstacleFamily}). Moreover, as $\epsilon$ approaches $0$ this family converges to the natural foliation of the unit disk described in Fig. \ref{fig:complete_foliation}(b) and given by
\begin{equation} \label{eq:optimalSolution}
\alpha_{\theta_f}(s) = \begin{cases}
(r^S_{D(\theta_f)}(s)\cos(\theta^S_{D(\theta_f)}(s)), r^S_{D(\theta_f)}(s)\sin(\theta^S_{D(\theta_f)}(s))) & \text{ if } \theta_f < 2\pi/3  \\
(r^W(s)\cos(\theta^W(s)), r^W(s)\sin(\theta^W(s))) & \text{ if } \theta_f \geq 2\pi /3
\end{cases}
\end{equation}
where $r^W(s), \theta^W(s)$ are the radial and angular coordinates of the unique smooth solution given by Eqs. (\ref{eq:rweak}-\ref{eq:tweak}) with $|B| = 1$ and $\theta_f$. This is made precise in the following proposition.

\begin{prop}
For $\theta_f \in (0, \pi)$, 
\begin{equation*}
\lim_{\epsilon \to 0}{d\left(\alpha^{\epsilon}_{\theta_f}, \alpha_{\theta_f}\right)} = 0,
\end{equation*}
where 
\begin{equation*}
d\left(\alpha_{\theta_f}, \alpha_{\theta_f}\right)=\sup_{0\leq s \leq 1}\inf_{0\leq t \leq 1}\left|\alpha^{\epsilon}_{\theta_f}(s)-\alpha_{\theta_f}(t)\right|
\end{equation*}
is the natural distance between the images of curves in the uniform norm. 
\end{prop}

\begin{proof}
Let $\theta_f \in (0, \pi)$ and define the sequence of functions $\alpha_{\theta_f}^\epsilon(s)$ by Eq. (\ref{eq:optimalConstrainedSolution}). 

\begin{enumerate}
\item If $\theta_f < 2\pi/3$, then there exists some $D(\theta_f) > 0$ such that the smooth solution $(r_{D(\theta_f)}^S(s), \theta_{D(\theta_f)}^S(s))$ given by Eqs. (\ref{eq:rparamaterization}-\ref{eq:thparamaterization}) terminates at $(1, \theta_f)$. From Proposition \ref{prop:characterization2} it follows that $\lim_{\epsilon\rightarrow 0}\theta_c^\epsilon  = \pi/3 \geq \theta_f/2$. Therefore, there exists $\epsilon^*$ such that $\epsilon < \epsilon^* \implies \alpha_{\theta_f}^\epsilon = \alpha_{\theta_f}$. Thus, for $\theta_f < 2\pi/3$
\begin{equation*}
\lim_{\epsilon \to 0}{d\left(\alpha_{\theta_f}^{\epsilon}, \alpha_{\theta_f}\right)} = 0.
\end{equation*}

\item If $\theta_f \geq 2\pi/3$,  $\alpha_{\theta_f}^\epsilon$ approaches the obstacle tangentially at $(r, \epsilon_c^\epsilon)$ along the path of a smooth solution for $s \in [0, 1/3]$. It follows from the convexity of smooth solutions that $\alpha_{\theta_f}^\epsilon([0, 1/3])$ is contained in the rectangular region $\mathcal{R}_\epsilon = \{(x, y) \in \mathbb{R}^2 : x \in [\epsilon\cos(\theta_c^\epsilon), 1], y \in [0, \epsilon\sin(\theta_c^\epsilon)]\}$. As $\epsilon \to 0$, $\mathcal{R}_\epsilon$ limits to the line $\{(x, y) \in \mathbb{R}^2 : x \in [0, 1], y = 0\}$. Moreover, $\alpha_{\theta_f}^\epsilon([1/3, 2/3])$ is on the obstacle and consequently limits to the origin as $\epsilon \to 0$. It follows immediately from radial symmetry that a rotated rectangular region can be constructed around $\alpha_{\theta_f}^\epsilon([2/3, 1])$ that limits to the line $\theta = \theta_f$. Hence each point on $\alpha_{\theta_f}^\epsilon$ limits to a point along the weak solution $(r^W(s)\cos(\theta^W(s)), r^W(s)\sin(\theta^W(s)))$ and thus for $\theta_f > 2\pi/3$
\begin{equation*}
\lim_{\epsilon \to 0}{d\left(\alpha_{\theta_f}^\epsilon, \alpha_{\theta_f}\right)} = 0.
\end{equation*}
\end{enumerate}
\end{proof}

\begin{figure}[ht!]
\centering
\begin{subfigure}[b]{0.45\textwidth}
\centering
\includegraphics[width=\textwidth]{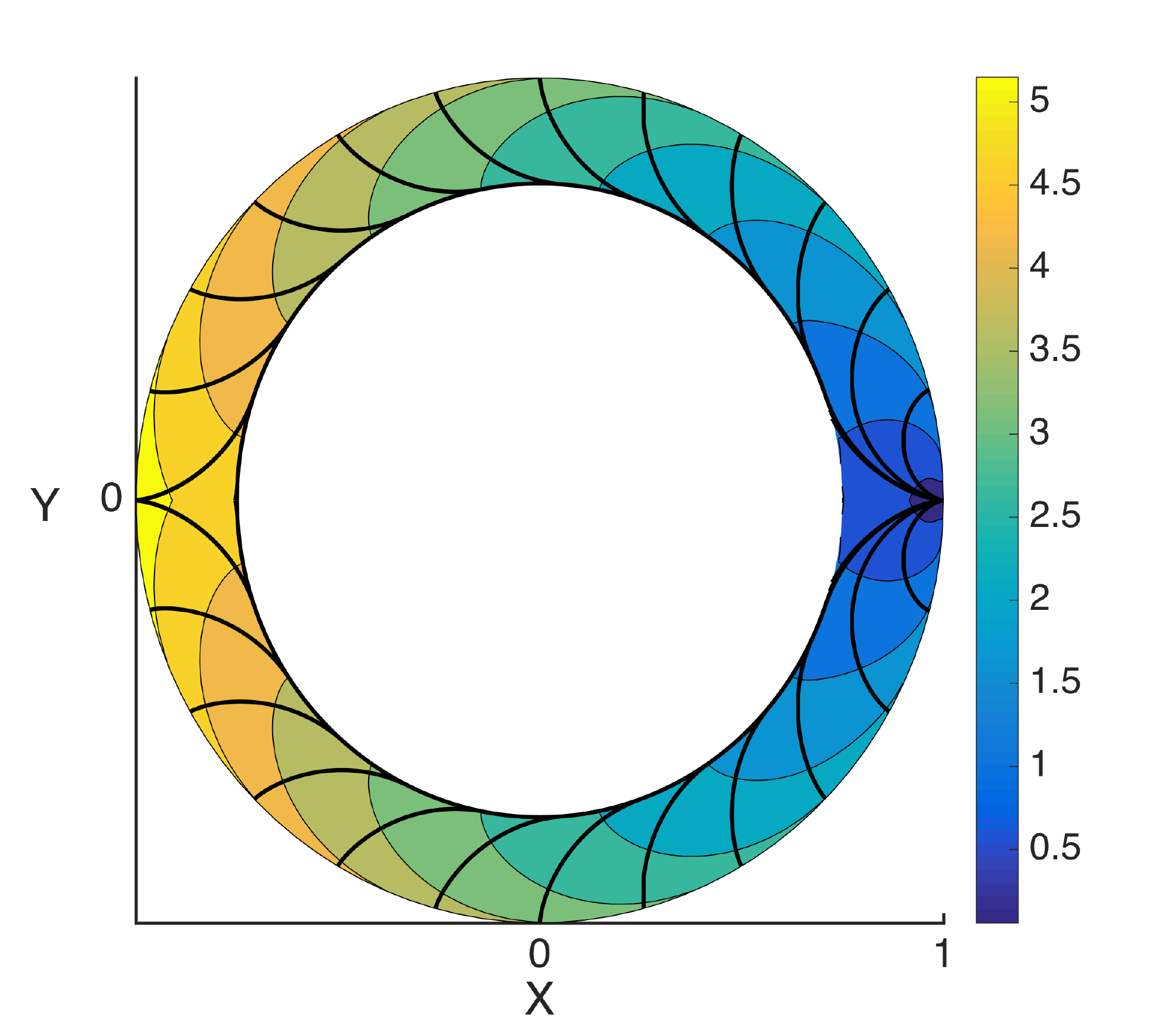}
\caption{}
\end{subfigure} 
~
\begin{subfigure}[b]{0.45\textwidth}
\centering
\includegraphics[width=\textwidth]{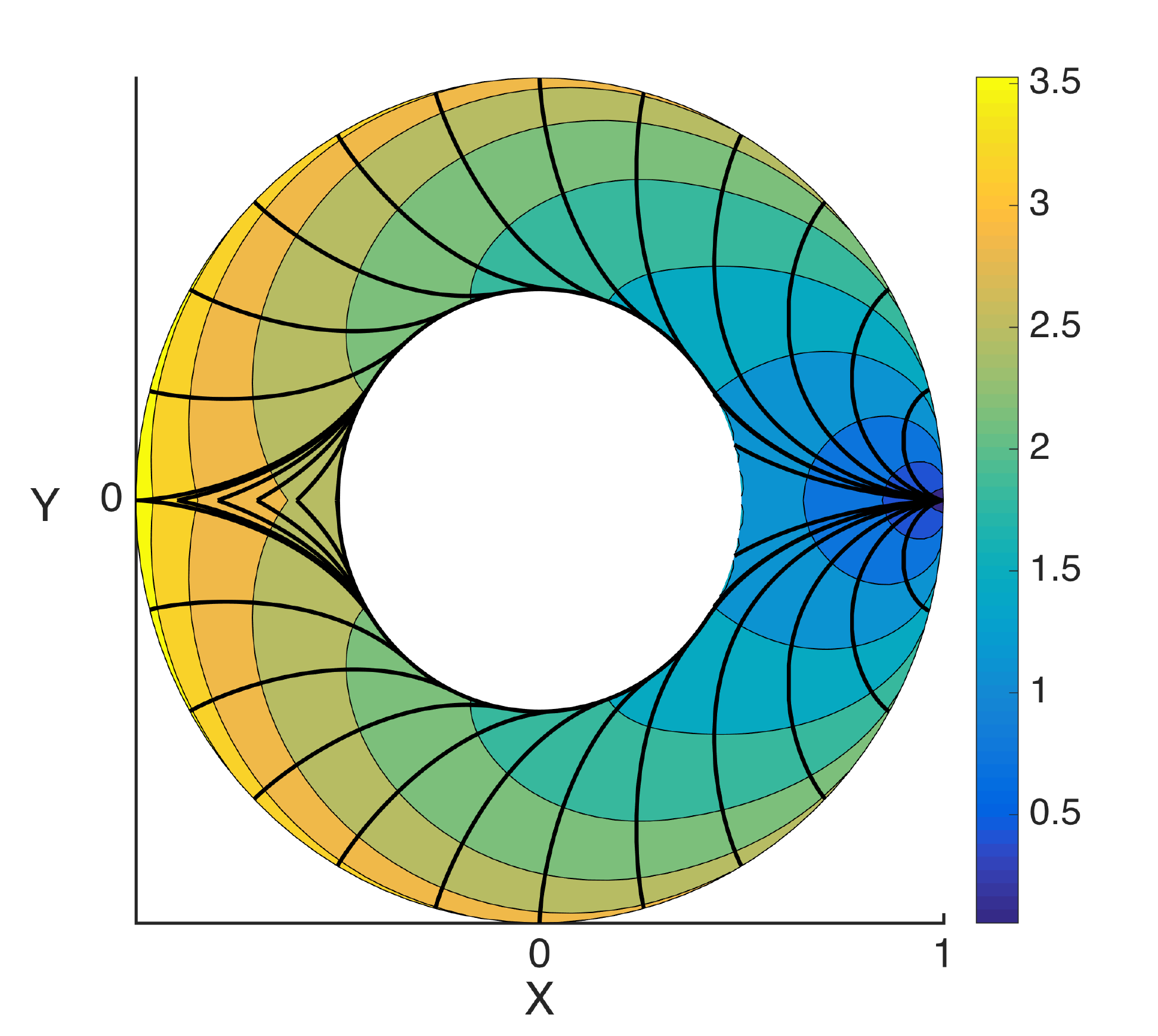}
\caption{}
\end{subfigure}
~
\begin{subfigure}[b]{0.45\textwidth}
\centering
\includegraphics[width=\textwidth]{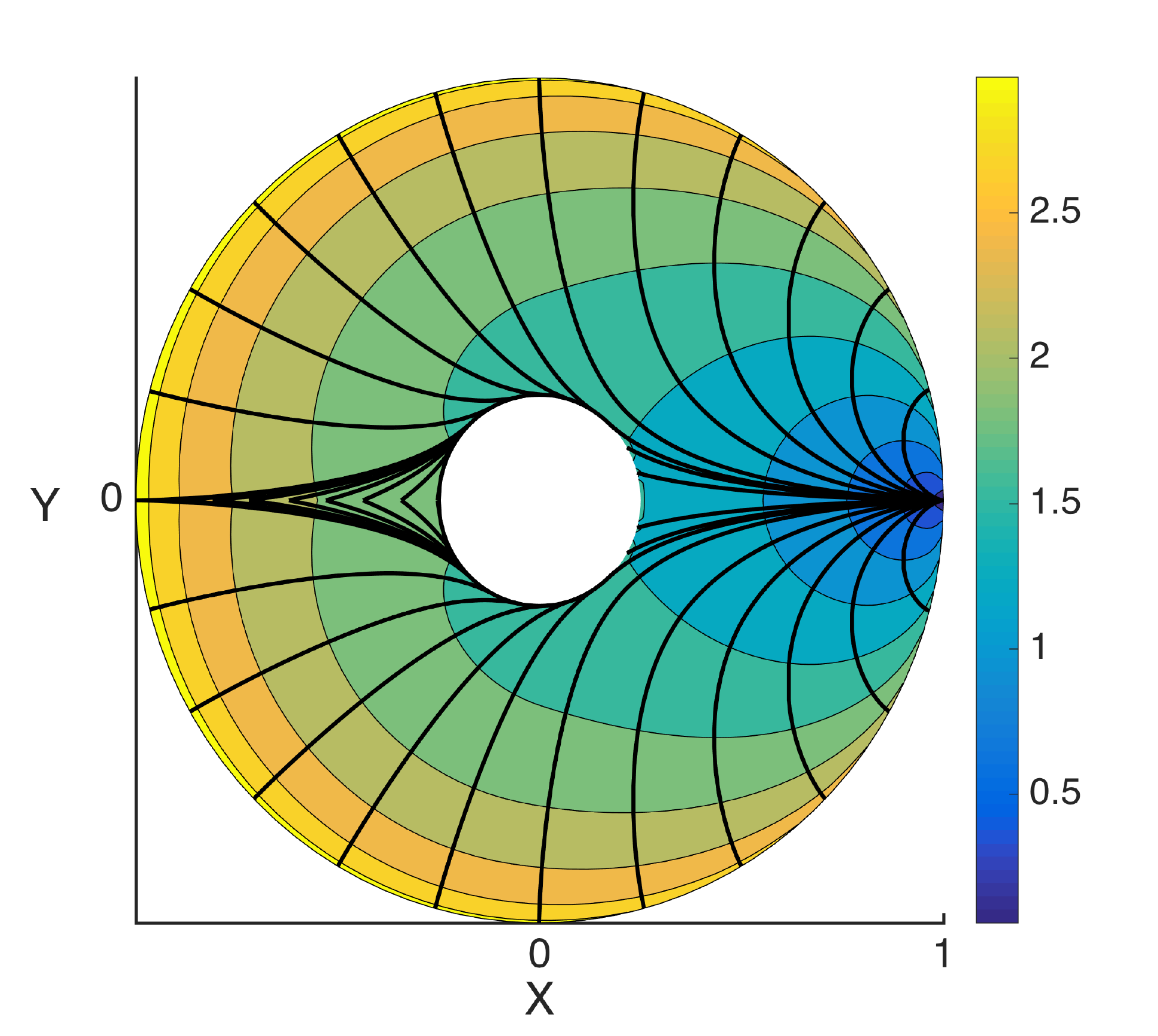}
\caption{}
\end{subfigure}
~
\begin{subfigure}[b]{0.45\textwidth}
\centering
\includegraphics[width=\textwidth]{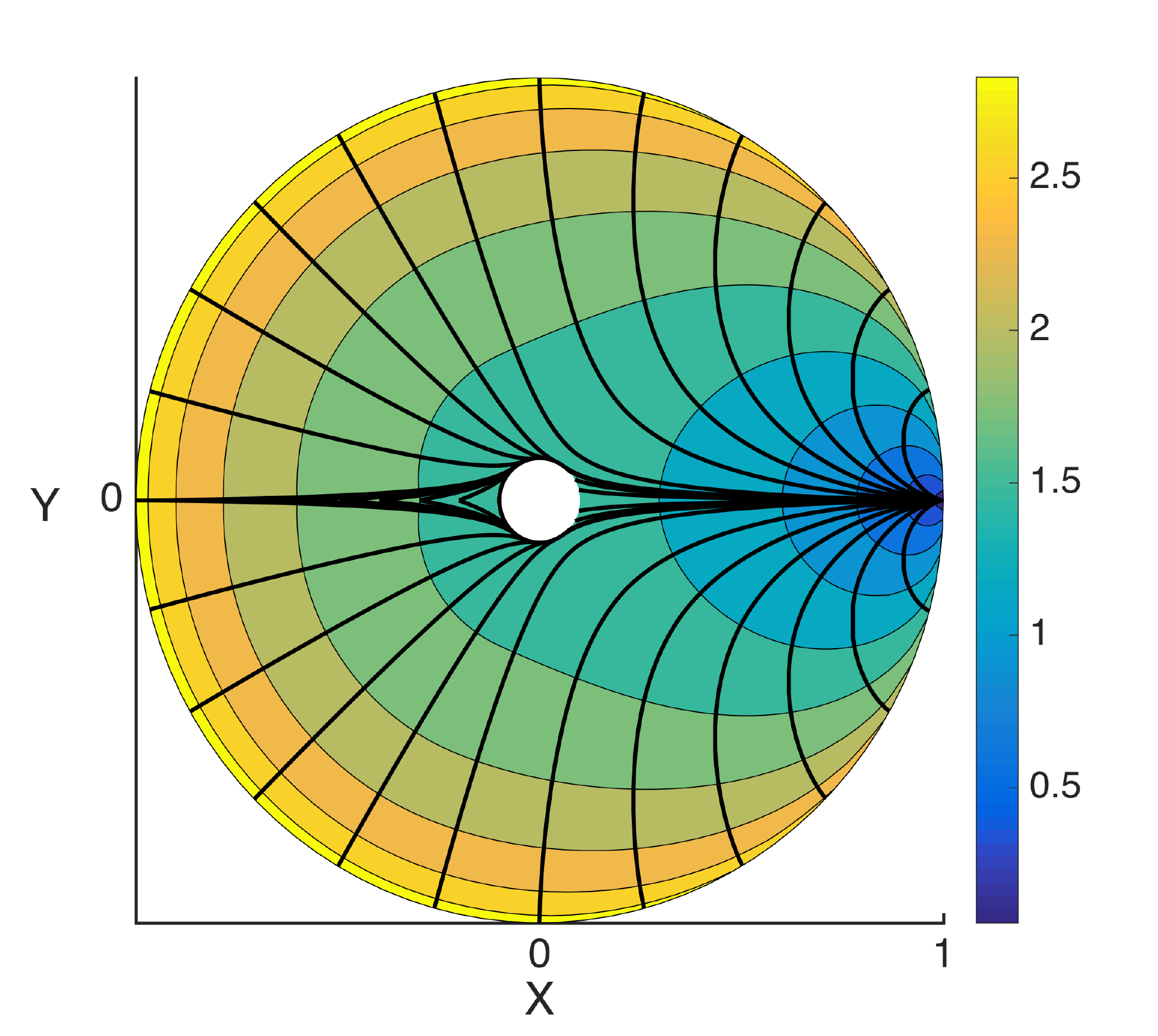}
\caption{}
\end{subfigure}

\caption{(a-d) A foliation of the annuli $\mathcal{O}_{0.75}$, $\mathcal{O}_{0.5}$, $\mathcal{O}_{0.25}$, $\mathcal{O}_{0.1}$ by evenly spaced solution curves terminating on the boundary of the annulus.  Beneath the solution curves of each subfigure is a contour plot of the time of flight from $(1, 0)$ to each point on the annulus by solution curves of the form given by Eq. (\ref{eq:optimalConstrainedSolution}).}
\label{fig:contourAnnulus}
\end{figure}

The solution curves depicted in Figure \ref{fig:contourAnnulus} again intersect the level-sets of the value function orthogonally. This is consistent with Eq. (\ref{Eq:Ray1}), i.e. $\mathcal{V}$ satisfies the eikonal equation defined by Eq. (\ref{Eq:Eikonal}) on an annular domain.

\section{Discussion and Conclusion}

In this paper we have solved the brachistochrone problem in the inverse-square gravitational field. Namely, we have constructed solutions that enter the so-called forbidden region first mentioned in \cite{parnovsky1998,Tee99,Gemmer2011}. Furthermore we considered the constrained problem where solutions are restricted to lie outside of a ball around the origin. This restricted problem is more physically relevant since it avoids the particle experiencing infinite acceleration at the origin. Moreover, the solutions in the annular domain recover our prior solutions on the disk in the limit of vanishing inner radius. Consequently these solutions on the annular domain correspond to ``regularized'' brachistochrone solutions that avoid the singularity. 

In the future, this work could be extended to problems with multiple singularities. That is, a natural extension of this work is to consider brachistochrone problems with multiple point sources of gravity. Natural questions to consider would be what role if any does the existence of a forbidden region play in the selection of strong or weak solutions. If weak solutions do exist, we conjecture that they would form a network of strong solutions patched together at singularities of the gravitational field. We expect that many of our results would hold locally near a singularity. However, by adding multiple singularities we break the radial invariance which we exploited to explicitly construct global solutions. 

We also should mention that we have only considered necessary conditions for optimality. Specifically, this problem is not completely solved in the modern sense without a proof of the existence of a minimizer. This is not a trivial task since the functional is not coercive and is not convex at the singular origin and hence the direct method of the calculus of variations cannot be applied. We conjecture that the general results for non-coercive integrals presented in \cite{botteron1991general} or the technique of convex rearrangement presented in \cite{greco2012minimization} can be adapted to prove existence on the annular domain. Consequently, we expect that we could prove an existence result on the entire disk by considering the limit of vanishing inner radius. 


\section*{Acknowledgments}
We like to acknowledge Zachary Nado for many fruitful conversations. CG would also like to thank the MAA for providing travel support to attend Mathfest 2015 where this work was first presented. JG is currently supported by NSF-RTG grant DMS-1148284.

\bibliographystyle{siamplain}
\bibliography{references}
\end{document}